\documentclass[11pt]{article}

\usepackage[utf8]{inputenc}
\usepackage[T1]{fontenc}
\usepackage[english]{babel}

\usepackage[a4paper,margin=1in]{geometry}
\usepackage{setspace}

\usepackage{amsmath,amssymb,amsfonts,amsthm,mathtools}
\usepackage{bm}
\usepackage{mathrsfs}

\usepackage{graphicx}
\usepackage{booktabs}
\usepackage{array}
\usepackage{enumitem}
\usepackage{caption}
\usepackage{subcaption}

\usepackage[numbers,sort&compress]{natbib}
\usepackage[colorlinks=true,
linkcolor=blue,
citecolor=blue,
urlcolor=blue]{hyperref}
\usepackage[nameinlink,capitalize]{cleveref}

\theoremstyle{plain}
\newtheorem{theorem}{Theorem}[section]
\newtheorem{proposition}[theorem]{Proposition}

\newtheorem{corollary}[theorem]{Corollary}

\theoremstyle{definition}
\newtheorem{definition}{Definition}
\newtheorem{example}{Example}

\theoremstyle{remark}
\newtheorem{remark}{Remark}

\newcommand{\R}{\mathbb{R}}

\newcommand{\D}{\mathbb{D}}
\newcommand{\Dpos}{\mathbb{D}^{+}}

\newcommand{\kh}{\mathbf{k}}

\newcommand{\eone}{\mathbf{e}_{1}}
\newcommand{\etwo}{\mathbf{e}_{2}}

\newcommand{\preceqD}{\preceq_{\mathbb{D}}}

\newcommand{\Dint}[2]{[#1,#2]_{\D}}

\newcommand{\Dnorm}[1]{\left\lVert #1\right\rVert_{\D}}
\newcommand{\Ddist}{d_{\D}}

\newcommand{\PD}{\mathcal{P}_{\mathbb{D}}}

\newcommand{\maj}{\succcurlyeq_{\mathrm{maj}}}
\newcommand{\Dmaj}{\succcurlyeq_{\mathbb{D},\mathrm{maj}}}

\newcommand{\Neg}{\mathcal{N}}

\newcommand{\Yager}{\mathcal{Y}}

\newcommand{\HD}{H_{\mathbb{D}}}
\newcommand{\GiniD}{G_{\mathbb{D}}}

\title{\textbf{On the Negation of a Hyperbolic-Valued Probability Distribution}}

\author{
	Juan Bory-Reyes,
	Edil D. Molina-Fernandez,
	Jos\'e M. Sigarreta-Almira
}

\date{}

\begin{document}
	
	\maketitle
	
	\begin{abstract}
In the context of hyperbolic numbers we define the concept of negation of finite hyperbolic-valued probability distributions that is based on the partial order induced by the idempotent structure of hyperbolic numbers. Then, a hyperbolic majorization and general hyperbolic negators are introduced. For a broad class of generated negators, we prove that the original distribution majorizes its negation. This comparison yields that entropy increase for the strong hyperbolic Shannon entropy and the hyperbolic Gini-Simpson entropy, and it implies component-wise uniformization of the iterated negation. Finally, we analyze involutive property of hyperbolic negators and prove that are structurally distinct from the generated negators responsible for the entropy increase. These results show that hyperbolic probabilistic negation is not merely a component-wise copy of the real case, but a theory governed by the interaction between idempotent decomposition, partial order, and entropy measure.
	\end{abstract}
	
	
	\noindent\textbf{Mathematic Subject Classification 2020:} 30G35; 94A17; 60A05; 06F25.

	\section{Introduction}
Hyperbolic numbers, also known as split-complex, double, Lorentz, or perplex numbers, constitute one of the simplest two-dimensional commutative extensions of the real line. Their origin goes back to the Cockle's work on tessarines in the nineteenth century \cite{Cockle1848}. Later developments connected them with the theory of functions of a hyperbolic complex variable, the geometry of the hyperbolic number plane, hypercomplex function theory, and Lorentzian models of special relativity \cite{Vignaux1935,Fjelstad1986,Sobczyk1995,Kisil2012,Sobczyk2013}. From an algebraic point of view, hyperbolic numbers can be considered as a hybrid between real and complex numbers. Their idempotent structure induce a natural component-wise partial order, which is of great value to discuss positivity, comparison, monotonicity, and order-reversing transformations in the hyperbolic setting. These features make the hyperbolic framework especially suitable for extending probabilistic and information-theoretic notions while preserving an intrinsic order structure.
	
	The axiomatic basis of probabilistic use of hyperbolic numbers has recently gained increasing attention. In \cite{Alpay2017} the authors introduced probability measures with values in the hyperbolic algebra as a direct analogue of Kolmogorov's axioms. Their construction shows that the basic rules of probability theory can be recovered in this setting, including conditional probability, the multiplication theorem, the law of total probability and Bayes' theorem. The idempotent decomposition of hyperbolic numbers yields a component-wise representation of the hyperbolic probability by two real probabilistic components, with the additional possibility that one component is inactive. This makes the hyperbolic probability simplex richer than the classical one, while preserving an intrinsic partial order suitable for comparison arguments. Subsequent works have used this framework to define hyperbolic extensions of Shannon entropy and related uncertainty measures, including weak and strong hyperbolic entropy, chaos-game constructions and Rényi-type entropy extensions \cite{TellezSanchez2021}. More recently \cite{Alpay2025}, probability measures with values in scaled hyperbolic numbers have also been developed, showing that this line of research still vivid.
	
	The systematic study of negations of finite probability distributions is traced to Yager's maximum-entropy negation, originally motivated by the problem of representing linguistic expressions such as ``not high'' when the positive concept is modeled by a probability distribution \cite{Yager2015}. The central idea is to redistribute the probability attached to an event among the remaining alternatives, so that the resulting distribution represents an opposite or negative informational state. This construction has been interpreted as an uncertainty-increasing transformation and has motivated several subsequent studies on the structure of probability negators. Early developments studied additional properties of Yager's transformation, including joint and marginal distributions, symmetry, convexity, bias, Markov-chain interpretations, information generation, and information loss \cite{Srivastava2017,Srivastava2021,Kaur2022,Klein2022}. In particular, later works introduced point-wise negators, distinguished between probability-distribution-independent and probability-distribution-dependent negators, characterized linear negators as combinations of Yager's negator and the uniform negator, and studied fixed points, contraction, involutivity, and convergence toward the uniform distribution \cite{Batyrshin2021,Batyrshin2021a,Batyrshin2021b,Batyrshin2023,EnsasteguiOrtega2024}.
	
	In subsequent research, a lot of attention has been drawn to analyze the uncertainty produced by negation through information-theoretic quantities, including Shannon-type measures, divergences, Tsallis entropy and related measures \cite{Srivastava2019,Zhang2019,Liu2021,Kaur2022}. Other contributions consider biased, exponential, generalized, or preference-based re-distributions of probability mass, as well as applications in sensor fusion, target recognition, decision making, and the estimation of distributions associated with negated properties \cite{Gao2019,Sun2020,Pham2021,Wu2022,Tanwar2023,Tanwar2023a}. Extensions to evidence theory have also been proposed through negations of basic probability assignments and belief structures \cite{Yin2019,Luo2020,Deng2020,Liu2023}. Recently, Xiao proposed a maximum-entropy negation for complex-valued distributions by introducing a complex-valued distribution model and an associated entropy functional \cite{Xiao2021}. These developments show that negation is a useful tool for representing uncertainty from an opposite perspective. However, the approach followed are not extended to the case of hyperbolic-valued probability distributions, where probabilities are governed by an intrinsic partial order and may have one or two active idempotent components. This leaves a natural question open about the intersection of negation, order, and hyperbolic-valued uncertainty.
	
	Motivated by this gap, the main goal of this paper is to develop a theory of negation for finite hyperbolic-valued probability distributions. The first step is to introduce a notion of hyperbolic majorization, which provides an intrinsic way to compare dispersion in the hyperbolic probability simplex. The second, is to define general hyperbolic negators as mass-preserving transformations that reverse the hyperbolic order whenever such an order comparison is available. These two notions allow us to formulate the central question of the paper in structural terms: under which conditions a hyperbolic-valued distribution is more concentrated than its negation? 
	
	 In this work, we provide a positive answer for this question by showing a broad class of generated hyperbolic negators, where such conditions are shown. We identify conditions on the generator that guarantee the required majorization comparison between a distribution and its negation. Once this comparison is established, the entropy results follow naturally from Schur-concavity. In particular, the strong hyperbolic Shannon entropy and the hyperbolic Gini--Simpson entropy increase under the corresponding negation, and the iterates are component-wise uniformized.
	
	This approach also clarifies why the hyperbolic framework is not a formal duplication of the real-valued theory. The idempotent decomposition makes component-wise arguments possible, but the presence of zero divisors, inactive components, and a partial order changes the meaning of negation. A hyperbolic negator must preserve the active mass structure and reverse order only where the order is meaningful. The  phenomena described do not appear in the classical setting, such as nonuniform fixed points for general negators and a structural distinction between entropy-increasing generated negators and involutive negators.
	
	The plan of the paper is as follows. Section~\ref{sec:hyperbolic-probability} recalls the basic facts on hyperbolic numbers and introduces finite hyperbolic-valued probability distributions. Section~\ref{sec:majorization-entropy} develops hyperbolic majorization and the entropy functionals used throughout the paper. Section~\ref{sec-hyperbolic-negators} introduces hyperbolic negators and studies structural properties, including generated negators and involutivity. Section~\ref{sec:entropy-convergence} proves the main comparison results and derives entropy increase and component-wise uniformization.

	\section{Hyperbolic-valued probability distributions}
	\label{sec:hyperbolic-probability}
	
	In this section we have complied some basic concepts, notations and geometric framework used throughout the paper. The idempotent representation of hyperbolic numbers allows us to treat a hyperbolic-valued probability distribution as a pair of real distributions, possibly with one inactive component. This component-wise viewpoint is essential and will be used later to define majorization, entropy measure, and negation. 
	
	\subsection{Hyperbolic numbers}
	
	We will touch only a few aspect of hyperbolic numbers. For more details on the algebraic and geometric structure we refer the reader to \cite{Sobczyk1995,Kisil2012,Vignaux1935, Sobczyk2013}.
	
	The set of hyperbolic numbers is the 2-dimensional associative, commutative real algebra spanned by $1$ and a distinct unit $\kh$ where $\kh^2=1$.
	\[
	\D=\{a+\kh b\mid a,b\in\R,\ \kh^2=1\}
	\]
	This algebra is not a field, since it contains non-trivial zero divisors. 
	
	The standard idempotent basis is given by
	\[
	\eone=\frac{1+\kh}{2},
	\qquad
	\etwo=\frac{1-\kh}{2}.
	\]
	These elements satisfy $\eone^2=\eone$, $\etwo^2=\etwo$, $\eone\etwo=0$, and $\eone+\etwo=1$.
	Therefore,
	\[
	\D=\R\eone+\R\etwo.
	\]
	Every $z=a+\kh b\in\D$ has a unique idempotent representation
	\[
	z=z^{(1)}\eone+z^{(2)}\etwo,
	\]
	where $z^{(1)}=a+b$, and $z^{(2)}=a-b$.
	Conversely,
	\[
	z^{(1)}\eone+z^{(2)}\etwo
	=
	\frac{z^{(1)}+z^{(2)}}{2}
	+
	\kh\frac{z^{(1)}-z^{(2)}}{2}.
	\]
	Addition and multiplication are component-wise in the idempotent basis
	\[
	(z^{(1)}\eone+z^{(2)}\etwo)
	+
	(w^{(1)}\eone+w^{(2)}\etwo)
	=
	(z^{(1)}+w^{(1)})\eone
	+
	(z^{(2)}+w^{(2)})\etwo,
	\]
	and
	\[
	(z^{(1)}\eone+z^{(2)}\etwo)
	(w^{(1)}\eone+w^{(2)}\etwo)
	=
	z^{(1)}w^{(1)}\eone
	+
	z^{(2)}w^{(2)}\etwo.
	\]
	
	Since $\eone\etwo=0$, the two idempotent axes contain all non-trivial
	zero divisors. More precisely, a non-zero element
	$z=z^{(1)}\eone+z^{(2)}\etwo$ is a zero divisor if and only if
	$z^{(1)}=0$ or $z^{(2)}=0$. We shall denote the corresponding null
	cone by $\mathfrak{G}=\R\eone\cup\R\etwo$.
	
	The hyperbolic numbers are endowed with the partial order
	\[
	z\preceqD w
	\quad\Longleftrightarrow\quad
	z^{(1)}\le w^{(1)}
	\ \text{and}\
	z^{(2)}\le w^{(2)}.
	\]
	In the basis $\{1,\kh\}$, write $z=a+\kh b$ and $w=c+\kh d$. Since
	\[
	a+\kh b=(a+b)\eone+(a-b)\etwo, \quad \text{and} \quad c+\kh d = (c+d)\eone + (c-d)\etwo.
	\]
	The preceding order can be equivalently written as
	\[
	z\preceqD w
	\quad\Longleftrightarrow\quad
	a+b\le c+d
	\ \text{and}\
	a-b\le c-d.
	\]
	By the cone of non-negative hyperbolic numbers we mean 
	\[
	\Dpos
	=
	\{z\in\D\mid 0\preceqD z\}
	=
	\{a+b\kh\in\D\mid a\ge |b|\}.
	\]
	In physics, this is the so-called "null cone" or "light cone".
	
	For $z,w\in\D$, with $z\preceqD w$, the hyperbolic interval from
	$z$ to $w$ is given by
	\[
	\Dint{z}{w}
	=
	\{\xi\in\D\mid z\preceqD \xi\preceqD w\}.
	\]
	In particular,
	\[
	\Dint{0}{1}
	=
	\{z^{(1)}\eone+z^{(2)}\etwo\mid
	0\le z^{(1)}\le 1,\ 0\le z^{(2)}\le 1\}.
	\]
	
	\begin{remark}
		The order $\preceqD$ is not total. For instance, $\eone$ and
		$\etwo$ are not comparable. Indeed,
		\[
		\eone=1\eone+0\etwo,
		\qquad
		\etwo=0\eone+1\etwo.
		\]
		The first component of $\eone$ is larger than the first
		component of $\etwo$. On the other hand, the second component of
		$\eone$ is smaller than the second component of $\etwo$. Therefore
		neither $\eone\preceqD\etwo$ nor $\etwo\preceqD\eone$ holds.
	\end{remark}
	
	\medskip
	
	We shall use the real-valued hyperbolic modulus induced by the
	quadratic form of the hyperbolic plane. For a hyperbolic number
	$z=a+\kh b$, we define
	\[
	|z|_\D
	=
	\sqrt{|a^2-b^2|}.
	\]
	In idempotent coordinates, this can be written as
	\[
	|z|_\D^2
	=|a^2-b^2| = |(a+b)(a-b)| =
	|z^{(1)}z^{(2)}|.
	\]
	This quantity is not a norm in the usual sense, since it vanishes on
	nonzero zero divisors. Nevertheless, it captures the intrinsic
	hyperbolic quadratic structure and will be used as a real-valued
	measure of separation.
	
	\subsection{Hyperbolic-valued probability distributions}
	
	Let $n \ge 2$ and let $\Omega=\{\omega_1,\ldots,\omega_n\}$ be a
	finite frame. For an $n$-tuple $P=(p_1,\ldots,p_n)\in\D^n$, we write
	\[
	p_j=p_j^{(1)}\eone+p_j^{(2)}\etwo,
	\qquad j=1,\ldots,n,
	\]
	and denote
	\[
	P^{(1)}=(p_1^{(1)},\ldots,p_n^{(1)}),
	\qquad
	P^{(2)}=(p_1^{(2)},\ldots,p_n^{(2)}).
	\]
	
	Using the scalar hyperbolic modulus component-wise, we define the associated
	quantity for hyperbolic-valued vectors by
	\[
	\Dnorm{P}
	=
	\left(
	\sum_{j=1}^{n}|p_j|_\D^2
	\right)^{1/2}
	=
	\sqrt{\sum_{j=1}^{n}
		\left|p_j^{(1)}p_j^{(2)}\right|}.
	\]
	For two hyperbolic-valued vectors $P,Q\in\D^n$, we write
	\[
	\Ddist(P,Q)=\Dnorm{P-Q}.
	\]
	This quantity is a hyperbolic separation rather than a metric. Indeed,
	$\Ddist(P,Q)=0$ if and only if, for every $j=1,\ldots,n$,
	\[
	p_j^{(1)}=q_j^{(1)}
	\quad\text{or}\quad
	p_j^{(2)}=q_j^{(2)}.
	\]
	Thus $\Ddist$ detects simultaneous variation in both idempotent components at the
	same coordinate. We shall use it as the real-valued separation induced by the
	quadratic form of $\D$.
	\medskip
	
	Following the finite version of the hyperbolic-valued probability framework introduced in \cite{Alpay2017}, the following definition extends the usual normalization of a probability distribution to the hyperbolic setting.
	
	\begin{definition}
		\label{def:hyperbolic-probability-distribution}
		An $n$-tuple $P=(p_1,\ldots,p_n)\in\D^n$ is called a hyperbolic-valued probability distribution on $\Omega$ if
		\[
		p_j\in\Dint{0}{1}
		\quad\text{for every } j=1,\ldots,n,
		\]
		and
		\[
		\sum_{j=1}^{n}p_j\in\{1,\eone,\etwo\}.
		\]
		The set of all hyperbolic-valued probability distributions on
		$\Omega$ will be denoted by $\PD(n)$. For $P\in\PD(n)$, its total
		mass is denoted by
		\[
		s(P)=\sum_{j=1}^{n}p_j
		=
		s^{(1)}(P)\eone+s^{(2)}(P)\etwo,
		\qquad
		s^{(\ell)}(P)=\sum_{j=1}^{n}p_j^{(\ell)},\quad \ell=1,2.
		\]
	\end{definition}
	
	The condition $s(P)\in\{1,\eone,\etwo\}$ gives three possible
	normalizations:
	\[
	\begin{array}{lcl}
		s(P)=1
		&\Longleftrightarrow&
		s^{(1)}(P)=1
		\ \text{and}\
		s^{(2)}(P)=1,\\[4pt]
		s(P)=\eone
		&\Longleftrightarrow&
		s^{(1)}(P)=1
		\ \text{and}\
		s^{(2)}(P)=0,\\[4pt]
		s(P)=\etwo
		&\Longleftrightarrow&
		s^{(1)}(P)=0
		\ \text{and}\
		s^{(2)}(P)=1.
	\end{array}
	\]
	We shall say that the $\ell$-th idempotent component of $P$ is
	active if $s^{(\ell)}(P)=1$, and inactive if $s^{(\ell)}(P)=0$.
	Equivalently, an active component is a real probability distribution, whereas
	an inactive component is identically zero.
	
	Thus, if $s(P)=1$, both idempotent components are active and $P$
	is equivalent to a pair of real probability distributions. If
	$s(P)=\eone$, only the first component is active, so $P\in(\R\eone)^n$;
	if $s(P)=\etwo$, only the second component is active, so
	$P\in(\R\etwo)^n$.
	
	This normalization can also be expressed in the basis $\{1,\kh\}$. If $p_j=a_j+\kh b_j$, then
	$p_j\in\Dint{0}{1}$ is equivalent to
	$0\le a_j+b_j\le1$ and $0\le a_j-b_j\le1$. Moreover,
	\[
	\begin{array}{lcl}
		s(P)=1
		&\Longleftrightarrow&
		\displaystyle \sum_{j=1}^{n}a_j=1
		\ \text{and}\
		\displaystyle \sum_{j=1}^{n}b_j=0,\\[6pt]
		s(P)=\eone
		&\Longleftrightarrow&
		\displaystyle \sum_{j=1}^{n}a_j=\frac12
		\ \text{and}\
		\displaystyle \sum_{j=1}^{n}b_j=\frac12,\\[6pt]
		s(P)=\etwo
		&\Longleftrightarrow&
		\displaystyle \sum_{j=1}^{n}a_j=\frac12
		\ \text{and}\
		\displaystyle \sum_{j=1}^{n}b_j=-\frac12.
	\end{array}
	\]
	
	The geometry of $\PD(n)$ is transparent in idempotent coordinates.
	Indeed, if
	\[
	\Delta_n=\left\{x=(x_1,\ldots,x_n)\in\R_{ \ge 0}^n\mid
	\sum_{j=1}^{n}x_j=1\right\},
	\]
	then the elements of $\PD(n)$ with total mass $s(P)=1$ are naturally
	identified with $\Delta_n\times\Delta_n$, since both idempotent
	components are active. On the other hand, the elements with
	$s(P)=\eone$ are identified with $\Delta_n$ through their first
	idempotent component, while the elements with $s(P)=\etwo$ are
	identified with another copy of $\Delta_n$ through their second
	idempotent component. These three components are disjoint in the 
	space $\D^n$, because they correspond to different total masses. More precisely,	$\PD(n)$ has the intrinsic decomposition
	\[
	\PD(n)
	=
	\bigsqcup_{s\in\{1,\eone,\etwo\}}
	\{P\in\PD(n)\mid s(P)=s\},
	\]
	where $\sqcup$ denotes disjoint union, and
	\[
	\begin{aligned}
		&\{P\in\PD(n)\mid s(P)=1\}\simeq \Delta_n \times \Delta_n, \\
		&\{P\in\PD(n)\mid s(P)=\eone\}\simeq \Delta_n, \\
		&\{P\in\PD(n)\mid s(P)=\etwo\}\simeq \Delta_n .
	\end{aligned}
	\]
	
	Figure~\ref{fig:PD2-geometry} illustrates this decomposition for
	$n=2$. Since $P=(p_1,p_2)$ is determined by $p_1$ once the total
	mass $s(P)$ is fixed, the figure represents the first entry
	$p_1=a+\kh b$ in the $(a,b)$-plane. The three panels are kept
	separate because they correspond to disjoint subsets of $\D^2$, even
	though their projections onto the $p_1$-plane may overlap.
	
	\begin{figure}[ht]
		\centering
		\includegraphics[width=\textwidth]{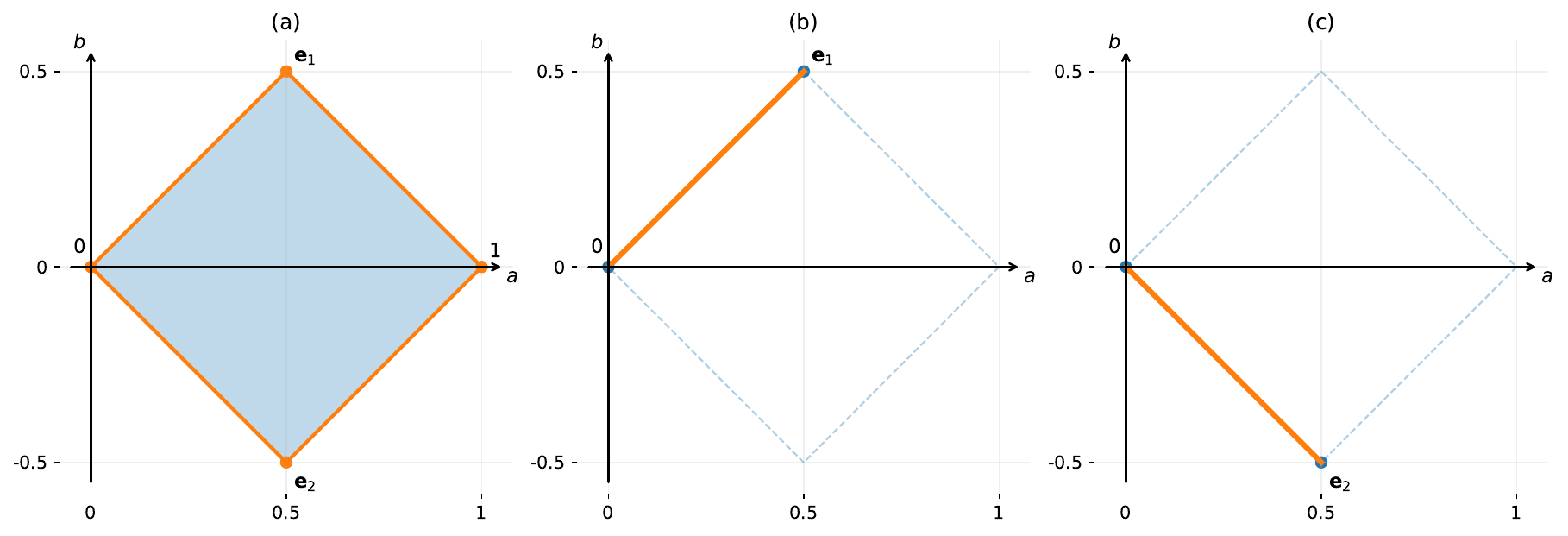}
		\caption{Intrinsic visualization of $\PD(2)$ in the
			$(a,b)$-plane through the first entry $p_1=a+\kh b$. Panel (a) corresponds to $s(P)=1$, where
			$p_1\in\Dint{0}{1}$ and the admissible region is the diamond with
			vertices $0,\eone,1,\etwo$. Panel (b) corresponds to
			$s(P)=\eone$, where $p_1$ lies on the segment joining $0$ and
			$\eone$. Panel (c) corresponds to $s(P)=\etwo$, where $p_1$
			lies on the segment joining $0$ and $\etwo$.}
		\label{fig:PD2-geometry}
	\end{figure}
	
	As in the classical simplex, the uniform distribution plays the role of the maximally spread distribution. In the hyperbolic setting, however, the uniform distribution must be defined relative to the total hyperbolic mass of the distribution.
	
	\begin{definition}
		\label{def:uniform-hyperbolic-distribution}
		For $s\in\{1,\eone,\etwo\}$, the uniform hyperbolic-valued
		probability distribution of length $n$ and total mass $s$ is
		\[
		U_n^s=
		\left(
		\frac{s}{n},\ldots,\frac{s}{n}
		\right)
		=		
		\left(
		\frac{s^{(1)}}{n},\ldots,\frac{s^{(1)}}{n}
		\right)\eone
		+
		\left(
		\frac{s^{(2)}}{n},\ldots,\frac{s^{(2)}}{n}
		\right)\etwo.
		\]
		In particular, the uniform distribution associated with
		$P\in\PD(n)$ is $U_n^{s(P)}$. Thus the three canonical uniform
		distributions in $\PD(n)$ are $U_n^1$, $U_n^{\eone}$, and
		$U_n^{\etwo}$.
	\end{definition}
	
	At the opposite extreme from the uniform distribution are the distributions concentrated at a single point in each active idempotent component. These are the hyperbolic analogues of the vertices of the classical probability simplex.
	
	Let $\delta_j\in\R^n$ denote the $j$-th point distribution on
	$\Omega$, that is,
	\[
	(\delta_j)_m=
	\begin{cases}
		1, & m=j,\\
		0, & m\neq j.
	\end{cases}
	\]
	
	\begin{definition}
		\label{def:hyperbolic-vertex-distribution}
		We say that $P\in\PD(n)$ is a hyperbolic vertex
		distribution if each active idempotent component is a real point
		distribution. Equivalently, for suitable
		$\alpha,\beta\in\{1,\ldots,n\}$,
		\[
		P=
		\begin{cases}
			\delta_{\alpha}\eone+\delta_{\beta}\etwo,
			& \text{if } s(P)=1,\\
			\delta_{\alpha}\eone,
			& \text{if } s(P)=\eone,\\
			\delta_{\beta}\etwo,
			& \text{if } s(P)=\etwo.
		\end{cases}
		\]
	\end{definition}
	
	The hyperbolic modulus captures specific information rather than measuring the size of each idempotent
	component separately, it measures how the two components interact. The next result shows how this interaction is bounded and identifies the extremal configurations.
	
	\begin{proposition}
		\label{prop:norm-bounds-hyperbolic-distributions}
		Let $n\ge 2$. If $P\in\PD(n)$, then
		\[
		0
		\le
		\Dnorm{P}^2
		\le
		s^{(1)}(P)s^{(2)}(P).
		\]
		In particular, $\Dnorm{P}=0$ for every $P$ with $s(P)\in\{\eone, \etwo\}$. If $s(P)=1$, the lower bound is attained if and only if the supports of
		$P^{(1)}$ and $P^{(2)}$ are disjoint. The upper bound is attained if and
		only if $P=\delta_\alpha\eone+\delta_\alpha\etwo$ for some $\alpha\in\{1,\ldots,n\}$.
	\end{proposition}
	
	\begin{proof}
		By definition of the hyperbolic modulus,
		\[
		\Dnorm{P}^2
		=
		\sum_{j=1}^{n}
		p_j^{(1)}p_j^{(2)}.
		\]
		The entries are non-negative, so the lower bound follows immediately.
		
		If one component is inactive, then either $p_j^{(1)}=0$ for every $j$ or
		$p_j^{(2)}=0$ for every $j$. Hence $\Dnorm{P}=0$ in both cases.
		
		It remains to consider the case $s(P)=1$. Then $P^{(1)}$ and $P^{(2)}$
		are real probability distributions. Therefore
		\[
		\sum_{j=1}^{n}
		p_j^{(1)}p_j^{(2)}
		\le
		\sum_{j=1}^{n}p_j^{(1)}
		=
		1.
		\]
		This proves the upper bound. The lower bound is attained precisely when
		$p_j^{(1)}p_j^{(2)}=0$ for every $j$, which is equivalent to the supports
		of $P^{(1)}$ and $P^{(2)}$ being disjoint.
		
		For the upper bound to be attained, we must have $p_j^{(2)}=1$
		whenever $p_j^{(1)}>0$. Since $P^{(2)}$ is a probability distribution,
		this forces $P^{(2)}=\delta_\alpha$ for some $\alpha$. The equality
		condition then forces $P^{(1)}=\delta_\alpha$ as well. The converse is immediate.
	\end{proof}
	
	We next compare a distribution with the uniform element through the hyperbolic separation. The estimate below expresses this comparison in terms of the intrinsic quadratic modulus and the total mass.
	
	\begin{proposition}
		\label{prop:distance-to-uniform}
		Let $n\ge 2$. If $P\in\PD(n)$, then
		\[
		\left|
		\Dnorm{P}^2
		-
		\frac{s^{(1)}(P)s^{(2)}(P)}{n}
		\right|
		\le
		\Ddist(P,U_n^{s(P)})^2
		\le
		\Dnorm{P}^2
		+
		3\frac{s^{(1)}(P)s^{(2)}(P)}{n}.
		\]
	\end{proposition}
	
	\begin{proof}
		For the lower bound, we have
		\[
		\begin{aligned}
			\Ddist(P,U_n^{s(P)})^2
			&=
			\sum_{j=1}^{n}
			\left|
			\left(
			p_j^{(1)}-\frac{s^{(1)}(P)}{n}
			\right)
			\left(
			p_j^{(2)}-\frac{s^{(2)}(P)}{n}
			\right)
			\right|\\
			&\ge
			\left|
			\sum_{j=1}^{n}
			\left(
			p_j^{(1)}-\frac{s^{(1)}(P)}{n}
			\right)
			\left(
			p_j^{(2)}-\frac{s^{(2)}(P)}{n}
			\right)
			\right|\\
			&=
			\left|
			\sum_{j=1}^{n}p_j^{(1)}p_j^{(2)}
			-
			\frac{s^{(1)}(P)}{n}\sum_{j=1}^{n}p_j^{(2)}
			-
			\frac{s^{(2)}(P)}{n}\sum_{j=1}^{n}p_j^{(1)}
			+
			\frac{s^{(1)}(P)s^{(2)}(P)}{n}
			\right|\\
			&=
			\left|
			\Dnorm{P}^2
			-
			\frac{s^{(1)}(P)s^{(2)}(P)}{n}
			-
			\frac{s^{(1)}(P)s^{(2)}(P)}{n}
			+
			\frac{s^{(1)}(P)s^{(2)}(P)}{n}
			\right|\\
			&=
			\left|
			\Dnorm{P}^2
			-
			\frac{s^{(1)}(P)s^{(2)}(P)}{n}
			\right|.
		\end{aligned}
		\]
		
		For the upper bound, we have
		\[
		\begin{aligned}
			\Ddist(P,U_n^{s(P)})^2
			&=
			\sum_{j=1}^{n}
			\left|
			p_j^{(1)}p_j^{(2)}
			-
			\frac{s^{(1)}(P)}{n}p_j^{(2)}
			-
			\frac{s^{(2)}(P)}{n}p_j^{(1)}
			+
			\frac{s^{(1)}(P)s^{(2)}(P)}{n^2}
			\right|\\
			&\le
			\sum_{j=1}^{n}p_j^{(1)}p_j^{(2)}
			+
			\frac{s^{(1)}(P)}{n}
			\sum_{j=1}^{n}p_j^{(2)}
			+
			\frac{s^{(2)}(P)}{n}
			\sum_{j=1}^{n}p_j^{(1)}
			+
			\sum_{j=1}^{n}
			\frac{s^{(1)}(P)s^{(2)}(P)}{n^2}\\
			&=
			\Dnorm{P}^2
			+
			\frac{s^{(1)}(P)s^{(2)}(P)}{n}
			+
			\frac{s^{(1)}(P)s^{(2)}(P)}{n}
			+
			\frac{s^{(1)}(P)s^{(2)}(P)}{n}\\
			&=
			\Dnorm{P}^2
			+
			\frac{3s^{(1)}(P)s^{(2)}(P)}{n}.
		\end{aligned}
		\]
	\end{proof}
	
	\section{Majorization order and hyperbolic entropy}
	\label{sec:majorization-entropy}
	
	Let us now introduce an order-theoretic tool used to compare the dispersion of two hyperbolic-valued probability distributions. Since hyperbolic probabilities decompose into idempotent components, majorization will be imposed separately on each active component. This allows classical Schur-concavity arguments to be transferred to the hyperbolic setting.
	
	Let $n\ge 2$, $x=(x_1,\ldots,x_n)\in\R^n$, and let
	$x^\downarrow=(x_1^\downarrow,\ldots,x_n^\downarrow)$ denote its
	decreasing rearrangement, so that
	$x_1^\downarrow \ge \cdots \ge x_n^\downarrow$.
	
	We use the standard notion of majorization for real vectors, as developed in the classical theory of inequalities and majorization \cite{Hardy1934, Marshall2011}. If
	$x,y\in\R^n$ are non-negative vectors, we say
	that $x$ majorizes $y$, and write $x\maj y$, if, for $k=1,\ldots,n-1$,
	\[
	\sum_{j=1}^{k}x_j^\downarrow
	\ge 
	\sum_{j=1}^{k}y_j^\downarrow,
	\qquad
	\text{and}
	\qquad
	\sum_{j=1}^{n}x_j=\sum_{j=1}^{n}y_j.
	\]
	
	The following definition is the natural component-wise extension of classical majorization. The requirement that the total hyperbolic masses coincide ensures that the same idempotent components are active in both distributions, so that each comparison takes place between real vectors with the same total mass.

	\begin{definition}
		\label{def:hyperbolic-majorization}
		Let $n\ge 2$, and $P,Q\in\PD(n)$ with $s(P)=s(Q)$. We say that $P$
		hyperbolically majorizes $Q$, and write $P\Dmaj Q$, if
		\[
		P^{(1)}\maj Q^{(1)}
		\quad\text{and}\quad
		P^{(2)}\maj Q^{(2)}.
		\]
	\end{definition}
	
	\begin{remark}
		The decreasing rearrangement is taken separately in each idempotent
		component. This is natural in the hyperbolic setting, since
		$\preceqD$ is only a partial order.
	\end{remark}
	
	The relation $\Dmaj$ inherits the basic preorder properties of real
	majorization in each idempotent component. Thus, for
	$P,Q,R\in\PD(n)$ with the same total mass, the following properties
	hold:
	\begin{enumerate}[label=\textup{(\roman*)}]
		\item $P\Dmaj P$.
		
		\item If $P\Dmaj Q$ and $Q\Dmaj R$, then $P\Dmaj R$.
		
		\item $P\Dmaj Q$ and $Q\Dmaj P$ if and only if $P^{(1)}$ is a
		permutation of $Q^{(1)}$, and $P^{(2)}$ is a permutation of
		$Q^{(2)}$.
	\end{enumerate}
	In particular, $\Dmaj$ becomes a partial order after identifying
	distributions whose idempotent components differ only by permutations.
	
	The basic geometry of majorization is preserved component-wise. Hyperbolic vertex distributions are the most concentrated elements, while the uniform distribution is the least concentrated element.
	
	\begin{proposition}
		\label{prop:vertices-and-uniform-majorization}
		Let $n\ge 2$, and $P\in\PD(n)$. If $V\in\PD(n)$ is a hyperbolic vertex
		distribution with $s(V)=s(P)$, then
		\[
		V\Dmaj P\Dmaj U_n^{s(P)}.
		\]
	\end{proposition}
	
	\begin{proof}
		Fix $\ell\in\{1,2\}$. If $s^{(\ell)}(P)=0$, then
		$V^{(\ell)}=P^{(\ell)}=(U_n^{s(P)})^{(\ell)}=0$, and the
		majorization relations are trivial in this component.
		
		Suppose now that $s^{(\ell)}(P)=1$. Then $P^{(\ell)}$ is a real
		probability distribution, $V^{(\ell)}=\delta_\alpha$ for some
		$\alpha\in\{1,\ldots,n\}$, and
		$(U_n^{s(P)})^{(\ell)}=(1/n,\ldots,1/n)$. Since
		$\delta_\alpha^\downarrow=(1,0,\ldots,0)$, for every
		$k=1,\ldots,n-1$,
		\[
		\sum_{j=1}^{k}(\delta_\alpha^\downarrow)_j
		=
		1
		\ge 
		\sum_{j=1}^{k}(P^{(\ell)})^\downarrow_j.
		\]
		Thus $\delta_\alpha\maj P^{(\ell)}$.
		
		On the other hand, because $(P^{(\ell)})^\downarrow$ is decreasing
		and has total mass $1$, the average of its first $k$ entries is at
		least the average of all its entries. Hence
		\[
		\frac1k
		\sum_{j=1}^{k}(P^{(\ell)})^\downarrow_j
		\ge 
		\frac1n,
		\qquad k=1,\ldots,n-1,
		\]
		or equivalently
		\[
		\sum_{j=1}^{k}(P^{(\ell)})^\downarrow_j
		\ge 
		\frac{k}{n}.
		\]
		Therefore $P^{(\ell)}\maj (1/n,\ldots,1/n)$. Applying these
		relations to both idempotent components gives
		$V\Dmaj P\Dmaj U_n^{s(P)}$.
	\end{proof}
	
	Majorization becomes useful for uncertainty once it is paired with Schur-concave functionals. A couple of examples will suffice to illustrate this point: the strong hyperbolic Shannon entropy in the setting of hyperbolic-valued probability distributions introduced in \cite{TellezSanchez2021} and the hyperbolic version of the quadratic Gini--Simpson entropy, obtained by a component-wise application of the classical Gini--Simpson functional \cite{Gini1912,Simpson1949}.
	
	\begin{definition}
		\label{def:hyperbolic-entropies}
		Let $n\ge 2$, and let $P=(p_1,\ldots,p_n)\in\PD(n)$, with
		$p_j=p_j^{(1)}\eone+p_j^{(2)}\etwo$. The strong hyperbolic Shannon
		entropy and the hyperbolic Gini--Simpson entropy of $P$ are defined by
		\[
		\begin{aligned}
			\HD(P)
			&=
			\left(
			-\sum_{j=1}^{n}p_j^{(1)}\log p_j^{(1)}
			\right)\eone
			+
			\left(
			-\sum_{j=1}^{n}p_j^{(2)}\log p_j^{(2)}
			\right)\etwo,\\[4pt]
			\GiniD(P)
			&=
			\left(
			s^{(1)}(P)
			-
			\sum_{j=1}^{n}\left(p_j^{(1)}\right)^2
			\right)\eone
			+
			\left(
			s^{(2)}(P)
			-
			\sum_{j=1}^{n}\left(p_j^{(2)}\right)^2
			\right)\etwo.
		\end{aligned}
		\]
		We use the convention $0\log 0=0$.
	\end{definition}
	
	The Gini--Simpson functional used here is the quadratic member of the
	Tsallis entropy family. Indeed, if $r=(r_1,\ldots,r_n)$ is a real
	probability vector, the Tsallis entropy is
	\[
	S_q(r)=\frac{1-\sum_{j=1}^{n}r_j^q}{q-1},
	\qquad q>0,\quad q\ne 1.
	\]
	Hence $S_2(r)=1-\sum_{j=1}^{n}r_j^2$, which is precisely the
	Gini--Simpson, or quadratic, entropy. Therefore, $\GiniD(P)$ is obtained
	by applying the $q=2$ Tsallis entropy to each active idempotent component
	of $P$
	\cite{Tsallis1988,Keylock2005}.
	
	The representation in the basis $\{1,\kh\}$ can be useful for
	interpreting these two entropies. Write
	\[
	\HD(P)=H^{(1)}(P)\eone+H^{(2)}(P)\etwo,
	\qquad
	\GiniD(P)=G^{(1)}(P)\eone+G^{(2)}(P)\etwo.
	\]
	Then
	\[
	\begin{aligned}
		\HD(P)
		&=
		\frac{H^{(1)}(P)+H^{(2)}(P)}{2}
		+
		\frac{H^{(1)}(P)-H^{(2)}(P)}{2}\kh,
		\\[4pt]
		\GiniD(P)
		&=
		\frac{G^{(1)}(P)+G^{(2)}(P)}{2}
		+
		\frac{G^{(1)}(P)-G^{(2)}(P)}{2}\kh.
	\end{aligned}
	\]
	Thus, for the proposed entropies, the coefficient of $1$ records the
	average of the two idempotent entropy values, while the coefficient of
	$\kh$ records their difference, up to the factor $1/2$.
	
	\begin{proposition}
		\label{prop:entropy-majorization-prelim}
		Let $n\ge 2$, and $P,Q\in\PD(n)$ with $s(P)=s(Q)$. If $P\Dmaj Q$, then
		\[
		\HD(P)\preceqD \HD(Q),
		\qquad
		\GiniD(P)\preceqD \GiniD(Q).
		\]
	\end{proposition}
	
	\begin{proof}
		The classical Shannon entropy and the classical Gini entropy are
		symmetric Schur-concave functions on the probability simplex. Hence,
		for every active idempotent component $\ell$,
		\[
		P^{(\ell)}\maj Q^{(\ell)}
		\quad\Longrightarrow\quad
		H(P^{(\ell)})\le H(Q^{(\ell)})
		\quad\text{and}\quad
		G(P^{(\ell)})\le G(Q^{(\ell)}).
		\]
		If the $\ell$-th component is inactive, then both
		$P^{(\ell)}$ and $Q^{(\ell)}$ are identically zero, because
		$s(P)=s(Q)$. Thus the same inequalities hold trivially in inactive
		components. The desired result follows from the component-wise order in
		$\D$.
	\end{proof}
	
	\begin{corollary}
		\label{cor:entropy-bounds}
		Let $n\ge 2$, and $P\in\PD(n)$. Then
		\[
		0\preceqD \HD(P)
		\preceqD
		s^{(1)}(P)\log(n)\eone
		+
		s^{(2)}(P)\log(n)\etwo,
		\]
		and
		\[
		0\preceqD \GiniD(P)
		\preceqD
		\left(s^{(1)}(P)\eone+s^{(2)}(P)\etwo\right)
		\left(1-\frac1n\right).
		\]
		In both cases, the upper bound is attained if and only if
		$P=U_n^{s(P)}$. The lower bound is attained if and only if $P$ is a
		hyperbolic vertex distribution.
	\end{corollary}
	
	\begin{proof}
		Let $V\in\PD(n)$ be a hyperbolic vertex distribution with
		$s(V)=s(P)$. By Proposition~\ref{prop:vertices-and-uniform-majorization},
		\[
		V\Dmaj P\Dmaj U_n^{s(P)}.
		\]
		Applying Proposition~\ref{prop:entropy-majorization-prelim} gives
		\[
		\HD(V)\preceqD \HD(P)\preceqD \HD(U_n^{s(P)}),
		\qquad
		\GiniD(V)\preceqD \GiniD(P)\preceqD \GiniD(U_n^{s(P)}).
		\]
		
		For every active component, $V^{(\ell)}$ is a real point
		distribution, so its Shannon and Gini entropies are both zero. Inactive
		components are also zero. Hence $\HD(V)=\GiniD(V)=0$.
		
		On the other hand, the active components of $U_n^{s(P)}$ are equal to
		$(1/n,\ldots,1/n)$. Therefore
		\[
		\HD(U_n^{s(P)})
		=
		s^{(1)}(P)\log(n)\eone
		+
		s^{(2)}(P)\log(n)\etwo
		\]
		and
		\[
		\GiniD(U_n^{s(P)})
		=
		\left(s^{(1)}(P)\eone+s^{(2)}(P)\etwo\right)
		\left(1-\frac1n\right).
		\]
		This proves the stated bounds. The equality cases follow from the
		classical equality cases; entropy is maximal only at the uniform
		probability distribution and minimal only at point distributions, applied to
		each active idempotent component.
		
	\end{proof}
	
	The next examples illustrate two different ways in which hyperbolic entropy may behave. In the first family, the two idempotent entropy components vary independently. In the second family, the two components are permutations of each other, and the hyperbolic entropy becomes real-valued.
	
	\begin{example}
		\label{ex:entropy-square-family}
		Let $0\le x,y\le1$, and consider
		\[
		P_{x,y}
		=
		\left(
		x\eone+y\etwo,\,
		(1-x)\eone,\,
		(1-y)\etwo
		\right).
		\]
		Then $P_{x,y}\in\PD(3)$ and $s(P_{x,y})=1$. Indeed,
		\[
		P_{x,y}^{(1)}=(x,1-x,0),
		\qquad
		P_{x,y}^{(2)}=(y,0,1-y).
		\]
		Thus both idempotent components are real probability distributions.
		
		The strong hyperbolic Shannon entropy is
		\[
		\HD(P_{x,y})=h(x)\eone+h(y)\etwo,
		\]
		where
		\[
		h(t)=-t\log t-(1-t)\log(1-t),
		\qquad 0\le t\le1.
		\]
		Similarly, the hyperbolic Gini
		entropy is
		\[
		\GiniD(P_{x,y})
		=
		2x(1-x)\eone+2y(1-y)\etwo.
		\]
		
		Hence the two idempotent components vary independently, as illustrated
		in Figure~\ref{fig:entropy-square-family}. Panels (a) and (c)
		correspond to the first idempotent component, namely
		$H^{(1)}(P_{x,y})=h(x)$ and $G^{(1)}(P_{x,y})=2x(1-x)$, and
		therefore depend only on $x$; accordingly, the corresponding
		surfaces are constant along the $y$-direction. Likewise, panels (b)
		and (d) correspond to the second idempotent component, namely
		$H^{(2)}(P_{x,y})=h(y)$ and $G^{(2)}(P_{x,y})=2y(1-y)$, and
		therefore depend only on $y$; hence these surfaces are constant
		along the $x$-direction.
		
		\begin{figure}[h]
			\centering
			\includegraphics[width=0.8\textwidth]{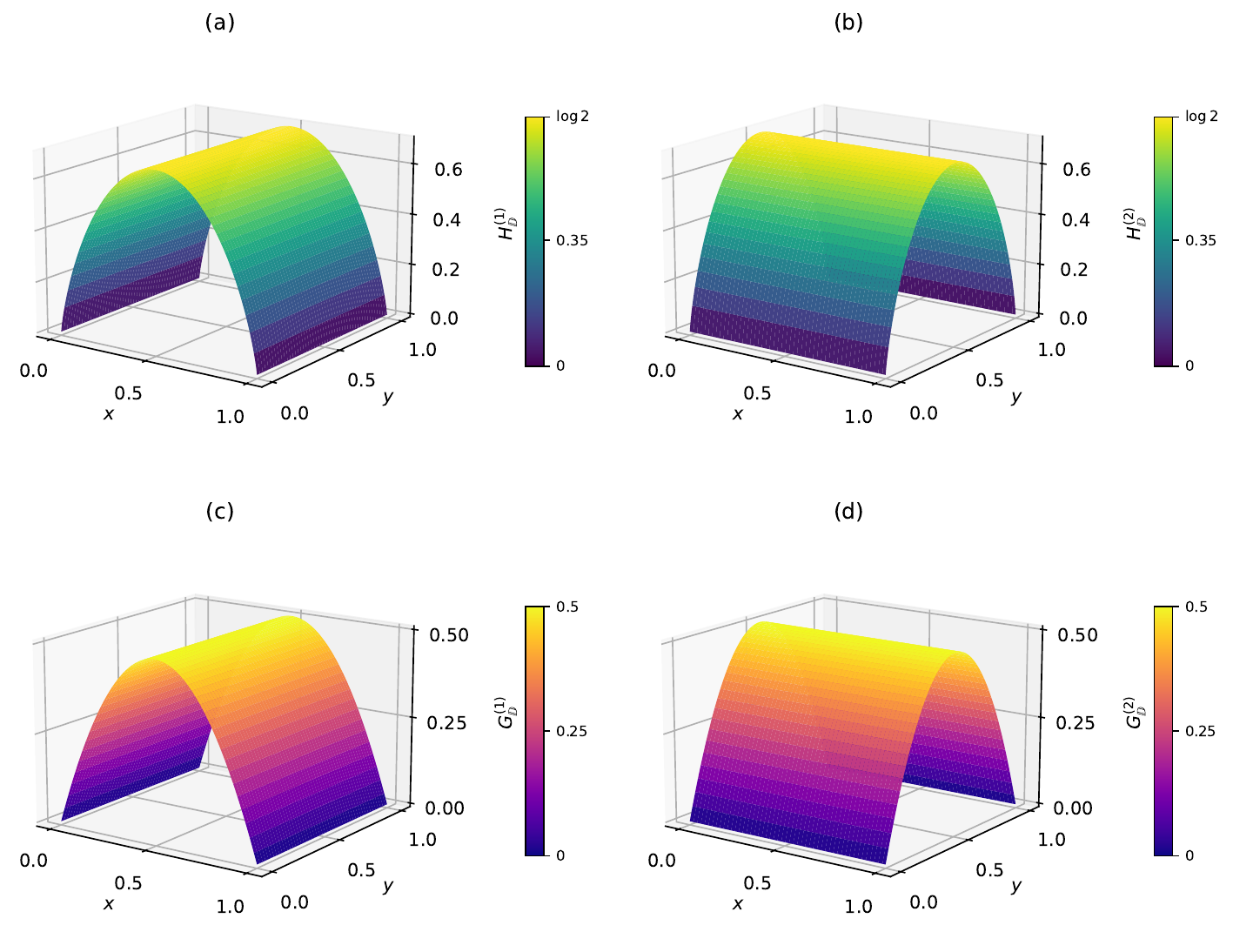}
			\caption{Entropy components for the family
				$P_{x,y}$. (a) $H^{(1)}(P_{x,y})=h(x)$;
				(b) $H^{(2)}(P_{x,y})=h(y)$;
				(c) $G^{(1)}(P_{x,y})=2x(1-x)$;
				(d) $G^{(2)}(P_{x,y})=2y(1-y)$.}
			\label{fig:entropy-square-family}
		\end{figure}
		
		The minimum is attained when $x,y\in\{0,1\}$, in which case
		$P_{x,y}$ is a hyperbolic vertex distribution. The maximum within
		this family is attained at $x=y=1/2$, where
		\[
		\HD(P_{1/2,1/2})=\log(2),
		\qquad
		\GiniD(P_{1/2,1/2})=\frac12.
		\]
	\end{example}
	
	\begin{example}
		\label{ex:entropy-triangular-family}
		Let
		\[
		T=\{(x,y)\in\R^2\mid x \ge 0,\ y \ge 0,\ x+y\le1\}.
		\]
		For $(x,y)\in T$, consider
		\[
		P_{x,y}
		=
		\left(
		x\eone+y\etwo,\,
		y\eone+(1-x-y)\etwo,\,
		(1-x-y)\eone+x\etwo
		\right).
		\]
		Then $P_{x,y}\in\PD(3)$ and $s(P_{x,y})=1$. Its idempotent
		components are
		\[
		P_{x,y}^{(1)}=(x,y,1-x-y),
		\qquad
		P_{x,y}^{(2)}=(y,1-x-y,x).
		\]
		Thus $P_{x,y}^{(2)}$ is a permutation of $P_{x,y}^{(1)}$. Since
		Shannon and Gini entropies are symmetric, both idempotent components
		have the same entropy.
		
		Therefore,
		\[
		\HD(P_{x,y})
		=
		H(x,y,1-x-y)\,(\eone+\etwo)
		=
		H(x,y,1-x-y),
		\]
		where
		\[
		H(x,y,1-x-y)
		=
		-x\log x-y\log y-(1-x-y)\log(1-x-y).
		\]
		Likewise,
		\[
		\GiniD(P_{x,y})
		=
		\left[
		1-x^2-y^2-(1-x-y)^2
		\right](\eone+\etwo).
		\]
		In this family the hyperbolic entropies are real-valued. The two
		surfaces are shown in Figure~\ref{fig:entropy-triangular-family}. Both
		entropies vanish at the vertices of $T$, where $P_{x,y}$ is a
		hyperbolic vertex distribution, and both are maximal at the barycenter
		$(x,y)=(1/3,1/3)$, where $P_{1/3,1/3}=U_3^1$. In particular,
		\[
		\HD(P_{1/3,1/3})=\log(3),
		\qquad
		\GiniD(P_{1/3,1/3})=\frac23.
		\]
		
		\begin{figure}[h]
			\centering
			\includegraphics[width=0.8\textwidth]{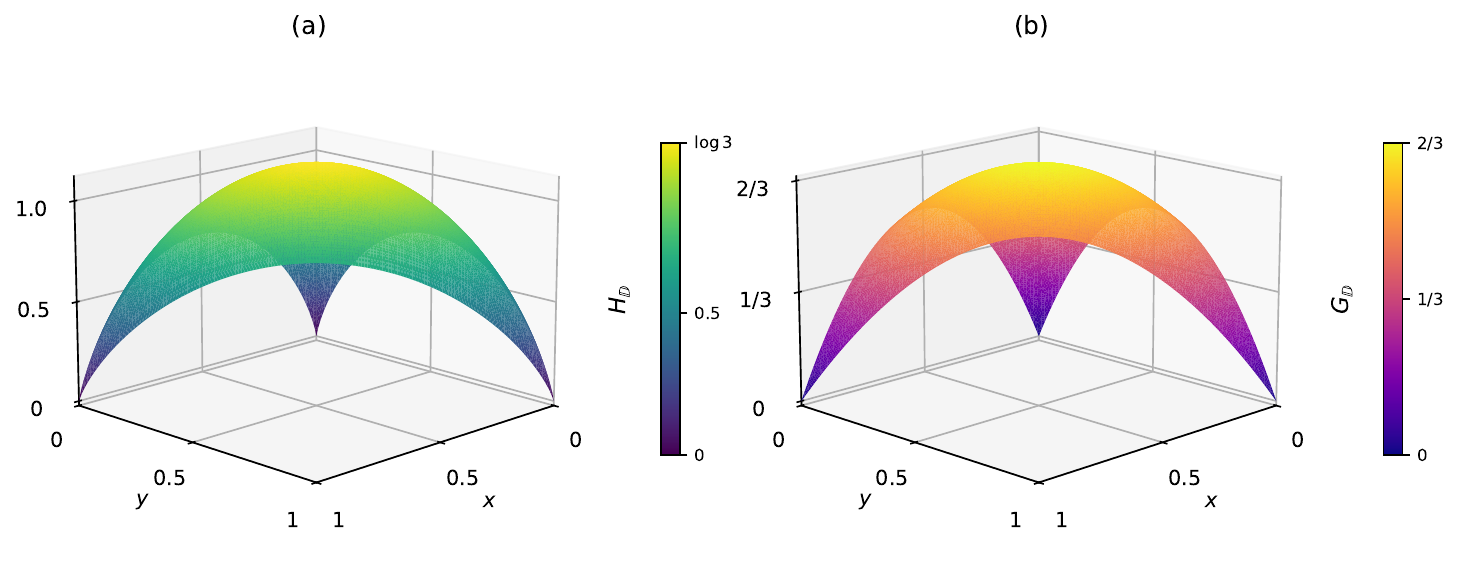}
			\caption{Entropy surfaces for the triangular family:
				(a) $\HD(P_{x,y})=H(x,y,1-x-y)$;
				(b) $\GiniD(P_{x,y})=1-x^2-y^2-(1-x-y)^2$.}
			\label{fig:entropy-triangular-family}
		\end{figure}
	\end{example}
	
	\section{Hyperbolic negators and structural properties}
	\label{sec-hyperbolic-negators}
	
	We now turn to the notion of negation. In the real-valued setting,
	Batyrshin et al.~\cite{Batyrshin2021a} distinguish probability
	transformations, negations of probability distributions, and pointwise
	negators generating such negations. We adopt the same scheme in the
	hyperbolic setting. 
	
	Since a hyperbolic probability distribution may have
	total mass $1$, $\eone$, or $\etwo$, probability transformations are
	required to preserve the idempotent total mass.
	
	\begin{definition}
		\label{def-hyperbolic-probability-transformation}
		Let $n\ge 2$. A map $\mathcal T:\PD(n)\to\PD(n)$ is called a
		hyperbolic probability transformation if, for every
		$P=(p_1,\ldots,p_n)\in\PD(n)$, with
		$\mathcal T(P)=Q=(q_1,\ldots,q_n)$, one has
		\[
		s(Q)=s(P), 
		\qquad \text{and} \qquad
		p_i=p_j
		\quad\Longrightarrow\quad
		q_i=q_j.
		\]
	\end{definition}
	
	We now introduce hyperbolic negators as the corresponding hyperbolic extension of the order-sensitive negations of real probability distributions introduced by Yager and further developed by Batyrshin et al.~\cite{Yager2015,Batyrshin2021a}. In the real case, if one coordinate is not larger than another, then the corresponding transformed coordinate is not smaller. In the hyperbolic case, this requirement is imposed with respect to the partial order $\preceqD$.
	
	\begin{definition}
		\label{def-hyperbolic-negator}
		Let $n\ge 2$. A hyperbolic probability transformation
		$\Neg:\PD(n)\to\PD(n)$ is called a hyperbolic negator if,
		for every $P=(p_1,\ldots,p_n)\in\PD(n)$, with
		$\Neg(P)=Q=(q_1,\ldots,q_n)$, one has
		\[
		p_i\preceqD p_j
		\quad\Longrightarrow\quad
		q_j\preceqD q_i.
		\]
	\end{definition}
	
	\begin{remark}
		Although equality preservation is included in the definition of a
		hyperbolic probability transformation, it is also forced by the
		order-reversing condition for negators. Indeed, if $p_i=p_j$, then
		$p_i\preceqD p_j$ and $p_j\preceqD p_i$. Hence $q_j\preceqD q_i$ and
		$q_i\preceqD q_j$, so $q_i=q_j$.
	\end{remark}
	
	Before introducing concrete families of negators, we point out a feature
	that is specific to the hyperbolic setting. In the real-valued case, the
	usual order on $[0,1]$ is total. Hence, if a real negator $N$ has a fixed
	point $P=(p_1,\ldots,p_n)$, then $N(P)=P$ and the negaotr condition imply that,
	whenever $p_i\le p_j$, one has $p_j\le p_i$. Since every pair of real
	coordinates is comparable, it follows that $p_i=p_j$ for all $i,j$.
	Thus, in this case, the only fixed point is the uniform distribution.
	
	The hyperbolic case is different because $\preceqD$ is only a partial
	order. Distinct hyperbolic probabilities may be incomparable, and the
	definition of a hyperbolic negator imposes no comparison condition on
	such pairs. Consequently, a general hyperbolic negator may admit
	non-uniform fixed points.
	
	\begin{example}
		Let $n=2$, and consider
		\[
		P=
		\bigl(
		\tfrac34\eone+\tfrac14\etwo,\,
		\tfrac14\eone+\tfrac34\etwo
		\bigr)
		\in\PD(2).
		\]
		The two entries of $P$ are incomparable with respect to $\preceqD$, since
		\[
		\tfrac34\eone+\tfrac14\etwo
		\not\preceqD
		\tfrac14\eone+\tfrac34\etwo,
		\qquad
		\tfrac14\eone+\tfrac34\etwo
		\not\preceqD
		\tfrac34\eone+\tfrac14\etwo.
		\]
		Define $\Neg:\PD(2)\to\PD(2)$ by
		\[
		\Neg(R)=
		\begin{cases}
			P, & R=P,\\[2pt]
			U_2^{s(R)}, & R\neq P.
		\end{cases}
		\]
		Then $\Neg$ preserves the total hyperbolic mass. If $R\neq P$, the image
		$\Neg(R)$ is uniform on its mass stratum, so the order-reversing condition is
		satisfied trivially. If $R=P$, the only comparable pairs among the two
		entries are the trivial ones. Hence the order-reversing condition is also
		satisfied. Therefore $\Neg$ is a hyperbolic negator.
		
		Moreover,
		\[
		\Neg(P)=P,
		\]
		so $P$ is a fixed point. However,
		\[
		P\neq U_2^1.
		\]
		Thus a general hyperbolic negator may have non-uniform fixed points. This
		phenomenon is caused by the partial nature of the hyperbolic order.
	\end{example}
	
	The previous definition is intentionally broad. In the real-valued theory, generated negators are obtained by applying a non-negative non-increasing function to the entries of a probability distribution and then normalizing the resulting weights \cite{Batyrshin2021,Batyrshin2021a}. We adapt this idea to the hyperbolic setting. To obtain negators with analyzable entropy behavior, we now introduce a generated class in which an order-reversing hyperbolic weight function is applied entrywise and the normalization is performed separately in each active idempotent component.
	
	\begin{definition}
		\label{def:admissible-hyperbolic-generator}
		Let $n\ge 2$. A map $\Phi:\Dint{0}{1}\longrightarrow \Dpos$, with
		\[
		\Phi(z)=\Phi^{(1)}(z)\eone+\Phi^{(2)}(z)\etwo,
		\]
		is called an admissible hyperbolic generator if the following
		conditions hold:
		\begin{enumerate}[label=\textup{(\roman*)}]
			\item $\Phi$ is order-reversing, that is,
			\[
			z\preceqD w
			\quad\Longrightarrow\quad
			\Phi(w)\preceqD \Phi(z).
			\]
			\item For every $P=(p_1,\ldots,p_n)\in\PD(n)$ and every active
			idempotent component $\ell$, one has
			\[
			Z_{\ell}^{\Phi}(P)
			=
			\sum_{j=1}^{n}\Phi^{(\ell)}(p_j)>0.
			\]
		\end{enumerate}
	\end{definition}
	
	Given an admissible generator, the associated transformation is obtained by normalizing the generated weights. The normalization is performed component-wise, so inactive components remain inactive and active components remain real probability distributions.
	
	\begin{definition}
		\label{def:generated-hyperbolic-transformation}
		Let $n\ge 2$, and let $\Phi:\Dint{0}{1}\to\Dpos$ be an admissible hyperbolic generator. The transformation
		generated by $\Phi$ is the map
		\[
		\Neg_{\Phi}:\PD(n)\longrightarrow\PD(n),
		\qquad
		\Neg_{\Phi}(P)=Q=(q_1,\ldots,q_n),
		\]
		defined component-wise as follows. For $j=1,\ldots,n$ and
		$\ell=1,2$,
		\[
		q_j^{(\ell)}
		=
		\begin{cases}
			\displaystyle
			\frac{\Phi^{(\ell)}(p_j)}
			{Z_{\ell}^{\Phi}(P)},
			& \text{if } s^{(\ell)}(P)=1,\\[10pt]
			0,
			& \text{if } s^{(\ell)}(P)=0.
		\end{cases}
		\]
		Equivalently,
		\[
		q_j=q_j^{(1)}\eone+q_j^{(2)}\etwo.
		\]
	\end{definition}
	
	\begin{proposition}
		\label{prop:generated-transformation-is-negator}
		Let $n\ge 2$, and $\Phi:\Dint{0}{1}\to\Dpos$ be an admissible hyperbolic generator. Then
		$\Neg_{\Phi}$ is a hyperbolic negator.
	\end{proposition}
	
	\begin{proof}
		Let $P=(p_1,\ldots,p_n)\in\PD(n)$, and let
		$\Neg_{\Phi}(P)=Q=(q_1,\ldots,q_n)$. Fix
		$\ell\in\{1,2\}$.
		
		If $s^{(\ell)}(P)=1$, then
		\[
		\sum_{j=1}^{n}q_j^{(\ell)}
		=
		\frac{1}{Z_{\ell}^{\Phi}(P)}
		\sum_{j=1}^{n}\Phi^{(\ell)}(p_j)
		=
		1.
		\]
		Moreover, $q_j^{(\ell)} \ge 0$ for every $j$. Hence the
		$\ell$-th component of $Q$ is a real probability distribution.
		
		If $s^{(\ell)}(P)=0$, then $q_j^{(\ell)}=0$ for every $j$. Thus
		the $\ell$-th component of $Q$ is inactive. Therefore
		$Q\in\PD(n)$ and $s(Q)=s(P)$.
		
		It remains to prove the order-reversing property. Suppose that
		$p_i\preceqD p_j$. Since $\Phi$ is order-reversing,
		\[
		\Phi(p_j)\preceqD \Phi(p_i).
		\]
		Thus, for each active component $\ell$,
		\[
		\Phi^{(\ell)}(p_j)\le \Phi^{(\ell)}(p_i).
		\]
		Dividing by the positive number $Z_{\ell}^{\Phi}(P)$ gives
		$q_j^{(\ell)}\le q_i^{(\ell)}$. Inactive components are identically
		zero. Hence $q_j\preceqD q_i$, and $\Neg_{\Phi}$ is a hyperbolic
		negator.
	\end{proof}
	
	The basic reference example in this class is the hyperbolic adaptation of Yager's real-valued negation of a probability distribution \cite{Yager2015}. The construction below is our component-wise version of that transformation and will serve as a reference map in the comparison results that follow.	
	
	\begin{example}
		\label{ex-yager-generator}
		Let $n\ge 2$, and consider the hyperbolic-valued function
		\[
		\Phi_Y:\Dint{0}{1}\longrightarrow \Dpos,
		\qquad
		\Phi_Y(z)=1-z.
		\]
		This function is order-reversing. Indeed, if $z\preceqD w$, then
		$1-w\preceqD 1-z$. Hence $\Phi_Y$ is an admissible hyperbolic
		generator.
		
		Let $P=(p_1,\ldots,p_n)\in\PD(n)$, and write $Q=\Neg_{\Phi_Y}(P)=(q_1,\ldots,q_n)$.	For every active idempotent component $\ell$, the normalization factor is
		\[
		Z_{\ell}^{\Phi_Y}(P)
		=
		\sum_{j=1}^{n}\left(1-p_j^{(\ell)}\right)
		=
		n-\sum_{j=1}^{n}p_j^{(\ell)}
		=
		n-1.
		\]
		Therefore,
		\[
		q_j
		=
		\frac{s^{(1)}(P)-p_j^{(1)}}{n-1}\eone
		+
		\frac{s^{(2)}(P)-p_j^{(2)}}{n-1}\etwo
		=
		\frac{s(P)-p_j}{n-1}.
		\]
		Thus
		\[
		\Neg_{\Phi_Y}(P)
		=
		\left(
		\frac{s(P)-p_1}{n-1},
		\ldots,
		\frac{s(P)-p_n}{n-1}
		\right).
		\]
		In this way we obtain what we shall call the Yager hyperbolic negator and denote by $\Yager$.
	\end{example}
	
	Involutivity is a classical property of logical negation and has also been studied for negations of probability distributions \cite{Batyrshin2021b,Batyrshin2023}. A natural question is whether generated hyperbolic negators are involutive. The next result shows that, under a simple visible obstruction at vertex distributions the involutivity fails. Thus, order reversal and normalization do not generally produce a reversible operation.
	
	\begin{theorem}
		\label{thm:vertex-obstruction-involutivity}
		Let $n\ge 2$, and $\Phi:\Dint{0}{1}\to\Dpos$ be an admissible hyperbolic
		generator, and let $\Neg_{\Phi}$ be the hyperbolic negator generated
		by $\Phi$. For $\ell=1,2$, set $Z_{\ell}^{\Phi} =  \Phi^{(\ell)}(1)+(n-1)\Phi^{(\ell)}(0)$,
		and define
		\[
		\phi_1^{\Phi}
		=
		\frac{\Phi^{(1)}(1)}{Z_1^{\Phi}}\eone
		+
		\frac{\Phi^{(2)}(1)}{Z_2^{\Phi}}\etwo,
		\qquad
		\phi_0^{\Phi}
		=
		\frac{\Phi^{(1)}(0)}{Z_1^{\Phi}}\eone
		+
		\frac{\Phi^{(2)}(0)}{Z_2^{\Phi}}\etwo.
		\]
		If $\Phi(\phi_0^{\Phi})\neq0$, then $\Neg_{\Phi}$ is not
		involutive.
	\end{theorem}
	
	\begin{proof}
		Consider the hyperbolic vertex distribution
		$V=(1,0,\ldots,0)\in\PD(n)$, which has total mass $s(V)=1$.
		Since $\Neg_{\Phi}$ is generated by $\Phi$, we have
		\[
		\Neg_{\Phi}(V)
		=
		(\phi_1^{\Phi},\phi_0^{\Phi},\ldots,\phi_0^{\Phi}).
		\]
		Indeed, in the $\ell$-th idempotent component, the corresponding
		normalizing factor is $Z_{\ell}^{\Phi}$.
		
		Let $\Neg_{\Phi}^2(V)=R=(r_1,\ldots,r_n)$. Then, for every
		$j\neq1$ and every $\ell\in\{1,2\}$,
		\[
		r_j^{(\ell)}
		=
		\frac{\Phi^{(\ell)}(\phi_0^{\Phi})}
		{\Phi^{(\ell)}(\phi_1^{\Phi})+(n-1)\Phi^{(\ell)}(\phi_0^{\Phi})}.
		\]
		If $\Phi(\phi_0^{\Phi})\neq0$, then
		$\Phi^{(\ell)}(\phi_0^{\Phi})>0$ for at least one idempotent
		component $\ell$. Hence $r_j^{(\ell)}>0$ for every $j\neq1$ in
		that component. Therefore $R\neq V$, and so
		$\Neg_{\Phi}^2(V)\neq V$. Thus $\Neg_{\Phi}$ is not involutive.
	\end{proof}
	
	\begin{corollary}
		\label{cor:zero-at-one-noninvolutive}
		Let $n>2$, and let $\Phi:\Dint{0}{1}\to\Dpos$ be an admissible
		hyperbolic generator such that
		\[
		\Phi(1)=0
		\qquad\text{and}\qquad
		\Phi\left(\frac{1}{n-1}\right)\neq0.
		\]
		Then the generated hyperbolic negator $\Neg_{\Phi}$ is not
		involutive.
	\end{corollary}
	
	\begin{proof}
		If $\Phi(1)=0$, then
		\[
		Z_{\ell}^{\Phi}
		=
		(n-1)\Phi^{(\ell)}(0),
		\qquad \ell=1,2.
		\]
		Since $\Phi$ is admissible, these normalizing factors are positive.
		Thus
		\[
		\phi_0^{\Phi}
		=
		\frac{1}{n-1}\eone+\frac{1}{n-1}\etwo
		=
		\frac{1}{n-1}.
		\]
		The result follows from
		Theorem~\ref{thm:vertex-obstruction-involutivity}.
	\end{proof}
	
	The previous obstruction does not mean that involutive hyperbolic
	negators cannot exist. It only shows that they are not generally obtained
	from a fixed generator of the preceding type. In the real-valued setting,
	Batyrshin introduced a probability-distribution-dependent involutive
	negator based on the maximum and minimum levels of the distribution
	\cite{Batyrshin2021b}. This idea was later used as a starting point for
	parametric families of probability negators \cite{Batyrshin2023}. We now
	construct a hyperbolic involutive negator adapted to the idempotent
	structure. The construction reverses the list of distinct probability
	levels in each active idempotent component and then normalizes
	component-wise.
	
	\begin{example}
		Let $n>2$. Consider first the Yager hyperbolic negator $\Yager$. Since
		$\Phi_Y(1)=0$ and
		\[
		\Phi_Y\left(\frac{1}{n-1}\right)
		=
		1-\frac{1}{n-1}
		=
		\frac{n-2}{n-1}
		\neq 0,
		\]
		Corollary~\ref{cor:zero-at-one-noninvolutive} states that $\Yager$ is
		not involutive. This obstruction is already visible at a vertex. If
		$V=(1,0,\ldots,0)\in\PD(n)$, then
		\[
		\Yager(V)
		=
		\left(0,\frac1{n-1},\ldots,\frac1{n-1}\right).
		\]
		For $n>2$, the value $1/(n-1)$ is not sent back to $0$ by the same
		generator. Hence $\Yager^2(V)\neq V$. For instance, when $n=3$ one has
		\[
		(1,0,0)
		\longmapsto
		\left(0,\frac12,\frac12\right)
		\longmapsto
		\left(\frac12,\frac14,\frac14\right).
		\]
		
		We now construct an involutive hyperbolic negator
		$M:\PD(n)\to\PD(n)$. Let $P=(p_1,\ldots,p_n)\in\PD(n)$, and fix an
		active component $\ell\in\{1,2\}$. Let $0\le a_1^{(\ell)}<a_2^{(\ell)}<\cdots<a_{k_\ell}^{(\ell)}\le 1$ be the distinct real values appearing among 	$p_1^{(\ell)},\ldots,p_n^{(\ell)}$. Let $\nu_r^{(\ell)}$ denote the
		multiplicity of the level $a_r^{(\ell)}$. If $k_\ell=1$, then $P^{(\ell)}$ is uniform on that active component,
		and we set $M(P)^{(\ell)}=P^{(\ell)}$. If $k_\ell\ge2$, for $r=1,\ldots,k_\ell$, define 
		\[
		\mu_{P,\ell}\left(a_r^{(\ell)}\right)
		=
		a_{k_\ell+1-r}^{(\ell)}.
		\]
		Thus the smallest level is sent to the largest one, the second smallest
		level to the second largest one, and so on. Set
		\[
		\Gamma_\ell(P)
		=
		\sum_{j=1}^{n}\mu_{P,\ell}\left(p_j^{(\ell)}\right)
		=
		\sum_{r=1}^{k_\ell}
		\nu_r^{(\ell)}a_{k_\ell+1-r}^{(\ell)}.
		\]
		Then, for $j=1,\ldots,n$, define
		\[
		M(P)_j^{(\ell)}
		=
		\frac{\mu_{P,\ell}\left(p_j^{(\ell)}\right)}
		{\Gamma_\ell(P)}.
		\]
		Inactive components are kept equal to zero. Hence
		\[
		M(P)_j=M(P)_j^{(1)}\eone+M(P)_j^{(2)}\etwo.
		\]
		
		The order used in this construction is the usual real order inside each
		active idempotent component. Thus, for a fixed component $\ell$, the
		condition $p_i^{(\ell)}\le p_j^{(\ell)}$ gives  $\mu_{P,\ell}(p_i^{(\ell)}) \ge \mu_{P,\ell}(p_j^{(\ell)})$. If the hyperbolic entries $p_1,\ldots,p_n$ are mutually comparable with
		respect to $\preceqD$, then the real component-wise rankings are
		compatible with a single ordering of the hyperbolic coordinates. In
		that case, the preceding level reversal agrees with the ordering used
		to form the decreasing rearrangements in the hyperbolic majorization
		order.

		The map $M$ is a hyperbolic negator if $p_i\preceqD p_j$ implies
		$p_i^{(\ell)}\le p_j^{(\ell)}$ in each active component. Then, we have,
		\[
		\mu_{P,\ell}\left(p_i^{(\ell)}\right)
		\ge
		\mu_{P,\ell}\left(p_j^{(\ell)}\right),
		\]
		and division by the positive number $\Gamma_\ell(P)$ yields
		\[
		M(P)_i^{(\ell)}\ge M(P)_j^{(\ell)}.
		\]
		Therefore $M(P)_j\preceqD M(P)_i$.
		
		It remains to verify involutivity. Fix an active component $\ell$. If
		$p_j^{(\ell)}=a_r^{(\ell)}$, then
		\[
		M(P)_j^{(\ell)}
		=
		\frac{a_{k_\ell+1-r}^{(\ell)}}{\Gamma_\ell(P)}.
		\]
		Hence the distinct levels of $M(P)^{(\ell)}$, written increasingly, are
		\[
		\frac{a_1^{(\ell)}}{\Gamma_\ell(P)}
		<
		\frac{a_2^{(\ell)}}{\Gamma_\ell(P)}
		<
		\cdots
		<
		\frac{a_{k_\ell}^{(\ell)}}{\Gamma_\ell(P)}.
		\]
		The multiplicity of $a_r^{(\ell)}/\Gamma_\ell(P)$ is
		$\nu_{k_\ell+1-r}^{(\ell)}$. Therefore
		\[
		\Gamma_\ell(M(P))
		=
		\sum_{r=1}^{k_\ell}
		\nu_{k_\ell+1-r}^{(\ell)}
		\frac{a_{k_\ell+1-r}^{(\ell)}}{\Gamma_\ell(P)}
		=
		\frac{1}{\Gamma_\ell(P)}
		\sum_{r=1}^{k_\ell}
		\nu_r^{(\ell)}a_r^{(\ell)}
		=
		\frac{1}{\Gamma_\ell(P)}.
		\]
		Thus,
		\[
		M^2(P)_j^{(\ell)}
		=
		\frac{
			\mu_{M(P),\ell}
			\left(
			a_{k_\ell+1-r}^{(\ell)}/\Gamma_\ell(P)
			\right)
		}
		{\Gamma_\ell(M(P))} 
		=
		\frac{a_r^{(\ell)}/\Gamma_\ell(P)}
		{1/\Gamma_\ell(P)}
		=
		a_r^{(\ell)}
		=
		p_j^{(\ell)}.
		\]
		This holds in every active component, while inactive components remain
		zero. Consequently, $M^2(P)=P$.
		
		For the visualization in Figure~\ref{fig:involutive-vs-yager}, we take
		$n=3$ and use two initial hyperbolic-valued probability distributions
		\[
		\begin{aligned}
			P_1&=(0.80,0.15,0.05)\eone+(0.10,0.75,0.15)\etwo,\\
			P_2&=(0.10,0.80,0.10)\eone+(0.05,0.35,0.60)\etwo.
		\end{aligned}
		\]
		
		For each $P_r$, and for each active idempotent component, the figure
		shows the two-step trajectory
		\[
		P_r\longmapsto N(P_r)\longmapsto N^2(P_r),
		\]
		where $N$ denotes either $\Yager$ or $M$, depending on the row of the
		figure. The upper panels correspond to the Yager hyperbolic negator.
		In those panels, the second iterate does not return to the initial
		point. The lower panels correspond to the level-generated negator $M$.
		In those panels, the second arrow returns to the initial point,
		illustrating $M^2(P_r)=P_r$. The left column shows the first
		idempotent component, while the right column shows the second
		idempotent component.
				
		\begin{figure}[h]
			\centering
			\includegraphics[width=0.8\textwidth]{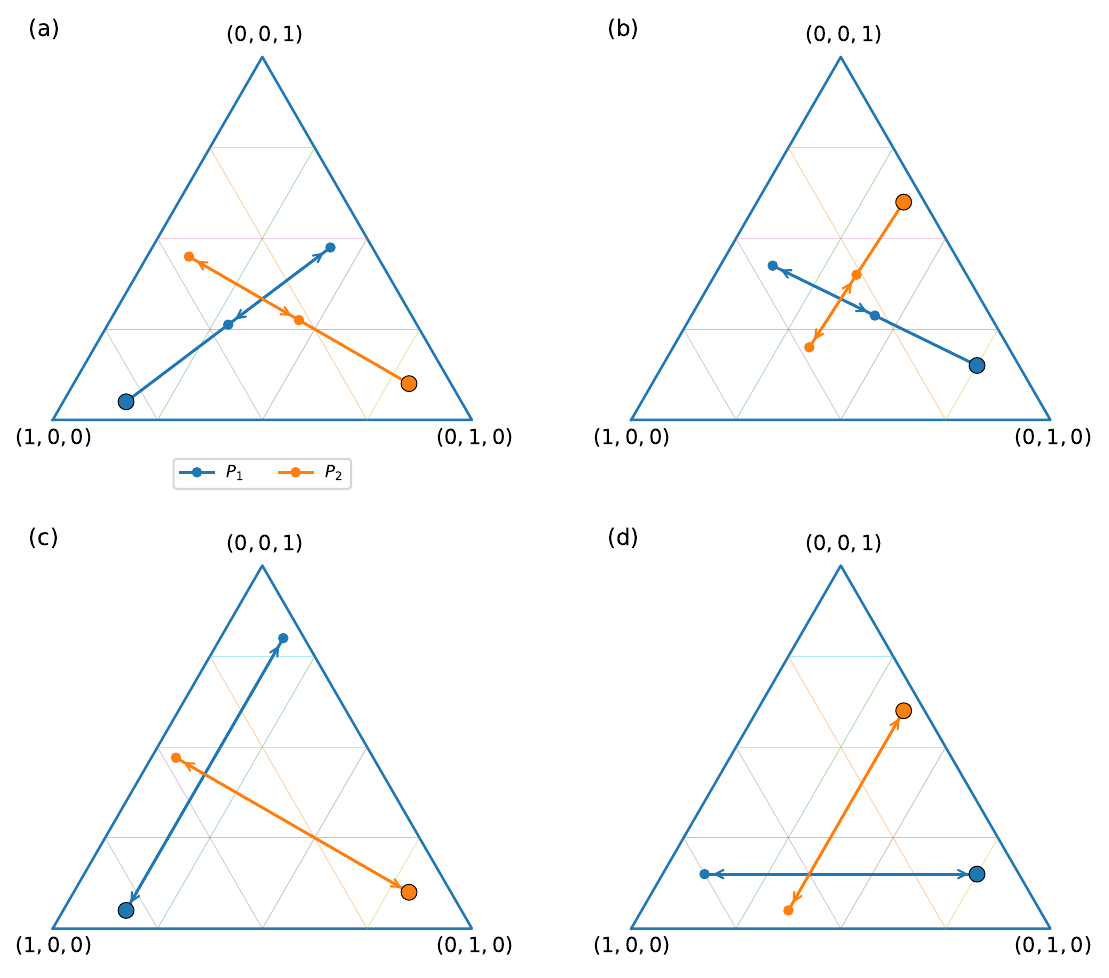}
			\caption{Two-step trajectories in the active components of $\PD(3)$.
				Panels (a) and (b) show the Yager hyperbolic negator. Panels (c)
				and (d) show the level-generated negator $M$. Panels (a) and (c)
				correspond to the first idempotent component, while panels (b)
				and (d) correspond to the second idempotent component. Each
				arrowed path represents $P\mapsto N(P)\mapsto N^2(P)$.}
			\label{fig:involutive-vs-yager}
		\end{figure}
		
		The action of $M$ can be read directly from the component-wise level
		reversal. For $P_1$, the first component has levels
		$0.05<0.15<0.80$, while the second component has levels
		$0.10<0.15<0.75$. In both components the normalization factor is equal
		to $1$. Hence
		\[
		\begin{aligned}
			M(P_1)
			&=
			(0.05,0.15,0.80)\eone
			+
			(0.75,0.10,0.15)\etwo,
			\\
			M^2(P_1)
			&=
			(0.80,0.15,0.05)\eone
			+
			(0.10,0.75,0.15)\etwo
			=
			P_1.
		\end{aligned}
		\]
		
		For $P_2$, the first component shows the role of normalization. Its
		levels are $0.10<0.80$, with multiplicities $2$ and $1$. Thus
		\[
		\Gamma_1(P_2)=2(0.80)+1(0.10)=1.70.
		\]
		The second component has levels $0.05<0.35<0.60$, and its normalization
		factor is equal to $1$. Therefore
		\[
		M(P_2)
		=
		\left(\frac{0.80}{1.70},\frac{0.10}{1.70},\frac{0.80}{1.70}\right)\eone
		+
		(0.60,0.35,0.05)\etwo.
		\]
		Applying $M$ once more gives
		\[
		M^2(P_2)
		=
		(0.10,0.80,0.10)\eone
		+
		(0.05,0.35,0.60)\etwo
		=
		P_2.
		\]
	\end{example}
	
	\section{Entropy increase and uniformization under hyperbolic negation}
	
	\label{sec:entropy-convergence}	
	
	The preceding section introduced the concept of hyperbolic negators. We now identify a structural condition under which a generated negator behaves as an uncertainty-increasing transformation. The key point is to prove that the original distribution majorizes its negation. Once this is established, entropy increase follows from Schur-concavity, and component-wise uniformization follows from a quantitative comparison with the Yager negator.
	
	We begin with a comparison principle. The Yager negator lies below the original distribution in the hyperbolic majorization order. Therefore, any negator that is further majorized by the Yager negator is automatically majorized by the original distribution.
	
	\begin{proposition}
		\label{prop:yager-intermediate-majorization}		
		Let $n\ge 2$. For every $P\in\PD(n)$, we have
		\[
		P\Dmaj \Yager(P).
		\]
		Moreover, if $\Neg:\PD(n)\to\PD(n)$ is a hyperbolic negator such
		that $\Yager(P)\Dmaj \Neg(P)$ for every $P\in\PD(n)$, then
		\[
		P\Dmaj \Neg(P).
		\]
	\end{proposition}
	
	\begin{proof}
		Fix an idempotent component $\ell\in\{1,2\}$. If
		$s^{(\ell)}(P)=0$, then both $P^{(\ell)}$ and
		$\Yager(P)^{(\ell)}$ are identically zero, and there is nothing to
		prove.
		
		Assume now that $s^{(\ell)}(P)=1$, and write $x=(P^{(\ell)})^\downarrow=(x_1,\ldots,x_n)$, with $x_1 \ge \cdots \ge x_n$.
		Since $P^{(\ell)}$ is a real probability distribution, one has
		$\sum_{i=1}^{n}x_i=1$. Moreover, because
		$t\mapsto (1-t)/(n-1)$ is decreasing,
		\[
		\left(\Yager(P)^{(\ell)}\right)^\downarrow
		=
		\left(
		\frac{1-x_n}{n-1},
		\frac{1-x_{n-1}}{n-1},
		\ldots,
		\frac{1-x_1}{n-1}
		\right).
		\]
		
		Let $k\in\{1,\ldots,n-1\}$ and $j\in\{1,\ldots,n\}$.
		For $j\le k$, using $x_j \ge x_k$ and
		$x_i\le x_k$ for $i=k+1,\ldots,n$, we get
		\[
		\sum_{i=k+1}^{n}x_i-x_j
		\le
		\sum_{i=k+1}^{n}x_k-x_k
		=
		(n-k-1)x_k.
		\]
		For $j>k$, using again $x_i\le x_k$ for
		$i=k+1,\ldots,n$, we get
		\[
		\sum_{i=k+1}^{n}x_i-x_j
		=
		\sum_{\substack{i=k+1\\ i\neq j}}^{n}x_i
		\le
		\sum_{\substack{i=k+1\\ i\neq j}}^{n}x_k
		=
		(n-k-1)x_k.
		\]
		Thus, in any case,
		\[
		\sum_{i=k+1}^{n}x_i-x_j
		\le
		(n-k-1)x_k.
		\]
		Therefore,
		\[
		\begin{aligned}
			\frac{1-x_j}{n-1}
			&=
			\frac{\sum_{i=1}^{k}x_i+\sum_{i=k+1}^{n}x_i-x_j}{n-1}
			\le
			\frac{\sum_{i=1}^{k}x_i+(n-k-1)x_k}{n-1}\\
			&=
			\frac{\sum_{i=1}^{k}(x_i-x_k)+(n-1)x_k}{n-1}\\
			&=
			\frac{
				k\sum_{i=1}^{k}(x_i-x_k)
				+k(n-1)x_k
				-(n-1)\sum_{i=1}^{k}x_i
				+(n-1)\sum_{i=1}^{k}x_i
			}{k(n-1)}\\
			&=
			\frac{
				k\sum_{i=1}^{k}(x_i-x_k)
				-(n-1)\sum_{i=1}^{k}(x_i-x_k)
				+(n-1)\sum_{i=1}^{k}x_i
			}{k(n-1)}\\
			&=
			\frac{
				-(n-1-k)\sum_{i=1}^{k}(x_i-x_k)
				+(n-1)\sum_{i=1}^{k}x_i
			}{k(n-1)}\\
			&\le
			\frac{(n-1)\sum_{i=1}^{k}x_i}{k(n-1)}=
			\frac1k\sum_{i=1}^{k}x_i.
		\end{aligned}
		\]
		Indeed, the last inequality follows from
		$n-1-k \ge 0$ and $\sum_{i=1}^{k}(x_i-x_k) \ge 0$.
		
		Now, using the decreasing rearrangement of
		$\Yager(P)^{(\ell)}$, we obtain
		\[
		\begin{aligned}
			\sum_{j=1}^{k}
			\left(\Yager(P)^{(\ell)}\right)^\downarrow_j
			=
			\sum_{j=n-k+1}^{n}\frac{1-x_j}{n-1}
			\le
			\sum_{j=n-k+1}^{n}
			\frac1k\sum_{i=1}^{k}x_i
			=
			\sum_{i=1}^{k}x_i
			=
			\sum_{i=1}^{k}(P^{(\ell)})^\downarrow_i .
		\end{aligned}
		\]
		Since $P^{(\ell)}$ and $\Yager(P)^{(\ell)}$ have the same total
		mass, it follows that
		\[
		P^{(\ell)}\maj \Yager(P)^{(\ell)}.
		\]
		
		Applying the argument to each active idempotent component gives
		$P\Dmaj\Yager(P)$.
		
		Now suppose that $\Neg:\PD(n)\to\PD(n)$ is a hyperbolic negator
		satisfying $\Yager(P)\Dmaj \Neg(P)$ for every $P\in\PD(n)$. Since
		hyperbolic majorization is transitive component-wise, we have
		\[
		P\Dmaj \Yager(P)\Dmaj \Neg(P).
		\]
		Therefore $P\Dmaj\Neg(P)$, as claimed.
	\end{proof}
	
	The following theorem provides a verifiable criterion for a generated negator satisfy this comparison principle. The quotient condition controls how the normalized weights produced by $\Phi$ compare with the affine weights of the Yager negator.
	
	\begin{theorem}
		\label{thm:generated-negator-yager-dominated}
		Let $n\ge 2$, and $\Phi:\Dint{0}{1}\to\Dpos$ be an admissible hyperbolic
		generator, and let $\Neg_{\Phi}$ be the generated hyperbolic
		negator. Suppose that, for every $P=(p_1,\ldots,p_n)\in\PD(n)$, every active
		idempotent component $\ell$, and all $i,j\in\{1,\ldots,n\}$,
		\[
		p_j^{(\ell)}=1
		\quad\Longrightarrow\quad
		\Phi^{(\ell)}(p_j)=0,
		\]
		and
		\[
		p_i^{(\ell)}\le p_j^{(\ell)}<1
		\quad\Longrightarrow\quad
		\Phi^{(\ell)}(p_i) \ge \Phi^{(\ell)}(p_j)
		\quad\text{and}\quad
		\frac{\Phi^{(\ell)}(p_i)}
		{1-p_i^{(\ell)}}
		\le
		\frac{\Phi^{(\ell)}(p_j)}
		{1-p_j^{(\ell)}}.
		\]
		Then, for every $P\in\PD(n)$, 
		\[
		\Yager(P)\Dmaj \Neg_{\Phi}(P), 
		\qquad \text{and} \qquad
		P\Dmaj \Neg_{\Phi}(P).
		\]
	\end{theorem}
	
	\begin{proof}
		Fix $P=(p_1,\ldots,p_n)\in\PD(n)$, and set
		$Q=\Neg_{\Phi}(P)$. We first prove that
		$\Yager(P)\Dmaj Q$.
		
		Let $\ell$ be an active idempotent component. Reorder the indices so
		that 
		$
		p_1^{(\ell)}\le p_2^{(\ell)}\le\cdots\le p_n^{(\ell)}.
		$
		Since the component is active, $\sum_{j=1}^{n}p_j^{(\ell)}=1$, and
		therefore
		\[
		\sum_{j=1}^{n}\left(1-p_j^{(\ell)}\right)=n-1.
		\]
		With this ordering, we have
		\[
		\left(\Yager(P)^{(\ell)}\right)^\downarrow
		=
		\left(
		\frac{1-p_1^{(\ell)}}{n-1},
		\ldots,
		\frac{1-p_n^{(\ell)}}{n-1}
		\right),
		\quad \text{and} \quad
		\left(Q^{(\ell)}\right)^\downarrow
		=
		\left(
		\frac{\Phi^{(\ell)}(p_1)}{Z_{\ell}^{\Phi}(P)},
		\ldots,
		\frac{\Phi^{(\ell)}(p_n)}{Z_{\ell}^{\Phi}(P)}
		\right).
		\]
		
		For $i\le k<j$, we have $p_i^{(\ell)}\le p_j^{(\ell)}$. If
		$p_j^{(\ell)}<1$, the quotient condition gives
		\[
		\frac{\Phi^{(\ell)}(p_i)}
		{1-p_i^{(\ell)}}
		\le
		\frac{\Phi^{(\ell)}(p_j)}
		{1-p_j^{(\ell)}}.
		\]
		Equivalently,
		\[
		\left(1-p_i^{(\ell)}\right)\Phi^{(\ell)}(p_j)
		-
		\Phi^{(\ell)}(p_i)\left(1-p_j^{(\ell)}\right)
		\ge 0.
		\]
		If $p_j^{(\ell)}=1$, then the boundary condition gives
		$\Phi^{(\ell)}(p_j)=0$, and the same expression is equal to $0$.
		Thus, for all $i\le k<j$,
		\[
		\left(1-p_i^{(\ell)}\right)\Phi^{(\ell)}(p_j)
		-
		\Phi^{(\ell)}(p_i)\left(1-p_j^{(\ell)}\right)
		\ge 0.
		\]
		Summing over $i=1,\ldots,k$ and $j=k+1,\ldots,n$, we obtain
		\[
		\begin{aligned}
			0
			&\le
			\sum_{i=1}^{k}\sum_{j=k+1}^{n}
			\left[
			\left(1-p_i^{(\ell)}\right)\Phi^{(\ell)}(p_j)
			-
			\Phi^{(\ell)}(p_i)\left(1-p_j^{(\ell)}\right)
			\right]\\
			&=
			\left(\sum_{i=1}^{k}\left(1-p_i^{(\ell)}\right)\right)
			\left(\sum_{j=k+1}^{n}\Phi^{(\ell)}(p_j)\right)
			-
			\left(\sum_{i=1}^{k}\Phi^{(\ell)}(p_i)\right)
			\left(\sum_{j=k+1}^{n}\left(1-p_j^{(\ell)}\right)\right)\\
			&=
			\left(\sum_{i=1}^{k}\left(1-p_i^{(\ell)}\right)\right)
			\left(\sum_{j=1}^{n}\Phi^{(\ell)}(p_j)\right)
			-
			\left(\sum_{i=1}^{k}\Phi^{(\ell)}(p_i)\right)
			\left(\sum_{j=1}^{n}\left(1-p_j^{(\ell)}\right)\right).
		\end{aligned}
		\]
		Since the component is active,
		\[
		\sum_{j=1}^{n}\left(1-p_j^{(\ell)}\right)=n-1.
		\]
		Moreover, the normalizing denominator
		$\sum_{j=1}^{n}\Phi^{(\ell)}(p_j)$ is positive. Hence the previous
		inequality gives
		\[
		\frac{\sum_{i=1}^{k}\Phi^{(\ell)}(p_i)}
		{\sum_{j=1}^{n}\Phi^{(\ell)}(p_j)}
		\le
		\frac{\sum_{i=1}^{k}\left(1-p_i^{(\ell)}\right)}
		{n-1}.
		\]
		
		Therefore,
		\[
		\sum_{i=1}^{k}
		\left(Q^{(\ell)}\right)^\downarrow_i
		\le
		\sum_{i=1}^{k}
		\left(\Yager(P)^{(\ell)}\right)^\downarrow_i,
		\qquad
		k=1,\ldots,n-1.
		\]
		Since $Q^{(\ell)}$ and $\Yager(P)^{(\ell)}$ have the same total
		mass, it follows that
		\[
		\Yager(P)^{(\ell)}\maj Q^{(\ell)}.
		\]
		Inactive components are identically zero, so
		\[
		\Yager(P)\Dmaj Q=\Neg_{\Phi}(P).
		\]
		
		Finally, by Proposition~\ref{prop:yager-intermediate-majorization}, we have $P\Dmaj\Neg_{\Phi}(P)$.
	\end{proof}
	
	The quotient condition in the previous theorem is abstract, but it is satisfied by a useful and natural class of component-wise generators. Concavity is the mechanism that forces the required monotonicity of the quotient.
	
	\begin{corollary}
		\label{cor:separable-concave-yager-dominated}
		Let $n\ge 2$, and $\Phi:\Dint{0}{1}\to\Dpos$ be an admissible hyperbolic
		generator of the form
		\[
		\Phi(z)=\varphi_1(z^{(1)})\eone+\varphi_2(z^{(2)})\etwo,
		\]
		where each $\varphi_\ell:[0,1]\to[0,\infty)$ is non-increasing,
		concave, and satisfies $\varphi_\ell(1)=0$. Then, for every
		$P\in\PD(n)$,
		\[
		\Yager(P)\Dmaj \Neg_{\Phi}(P), 
		\qquad \text{and} \qquad
		P\Dmaj \Neg_{\Phi}(P).
		\]
	\end{corollary}
	
	\begin{proof}
		Fix $P=(p_1,\ldots,p_n)\in\PD(n)$, an active idempotent component
		$\ell$, and indices $i,j\in\{1,\ldots,n\}$. Thus, $\Phi^{(\ell)}(p_j)=\varphi_\ell(p_j^{(\ell)})$.
		
		If $p_j^{(\ell)}=1$, then
		\[
		\Phi^{(\ell)}(p_j)=\varphi_\ell(1)=0.
		\]
		Now suppose that $p_i^{(\ell)}\le p_j^{(\ell)}<1$. Since
		$\varphi_\ell$ is non-increasing,
		\[
		\Phi^{(\ell)}(p_i)
		=
		\varphi_\ell(p_i^{(\ell)})
		\ge 
		\varphi_\ell(p_j^{(\ell)})
		=
		\Phi^{(\ell)}(p_j).
		\]
		
		It remains to verify the quotient condition. Define
		\[
		\psi_\ell(t)=\varphi_\ell(1-t),
		\qquad 0\le t\le1.
		\]
		Then $\psi_\ell$ is concave and satisfies $\psi_\ell(0)=0$. Hence,
		for $0<u\le v$,
		\[
		\frac{\psi_\ell(v)}{v}
		\le
		\frac{\psi_\ell(u)}{u}.
		\]
		Taking $u=1-p_j^{(\ell)}$ and $v=1-p_i^{(\ell)}$, we obtain
		\[
		\frac{\varphi_\ell(p_i^{(\ell)})}{1-p_i^{(\ell)}}
		\le
		\frac{\varphi_\ell(p_j^{(\ell)})}{1-p_j^{(\ell)}}.
		\]
		Thus the hypotheses of
		Theorem~\ref{thm:generated-negator-yager-dominated} are satisfied and the results follows.
	\end{proof}	
	
	We can now translate the majorization comparison into entropy. Since both entropy functionals are Schur-concave in each active component, the relation $P\Dmaj \Neg_\Phi(P)$ implies that the generated negation cannot decrease uncertainty.
	
	\begin{theorem}
		\label{thm:generated-negator-increases-entropy}
		Let $n\ge 2$, and $\Phi:\Dint{0}{1}\to\Dpos$ be an admissible hyperbolic
		generator satisfying the hypotheses of
		Theorem~\ref{thm:generated-negator-yager-dominated}. Then, for every
		$P\in\PD(n)$,
		\[
		\HD(P)\preceqD \HD(\Neg_\Phi(P)),
		\qquad \text{and} \qquad
		\GiniD(P)\preceqD \GiniD(\Neg_\Phi(P)).
		\]
	\end{theorem}
	
	\begin{proof}
		By Theorem~\ref{thm:generated-negator-yager-dominated}, one has
		$P\Dmaj \Neg_\Phi(P)$ for every $P\in\PD(n)$. Proposition~\ref{prop:entropy-majorization-prelim} gives
		\[	
		\HD(P)\preceqD \HD(\Neg_\Phi(P)),
		\qquad
		\GiniD(P)\preceqD \GiniD(\Neg_\Phi(P)).
		\]
	\end{proof}
	
	The comparison principle also gives a quantitative uniformization mechanism.
	For each distribution we measure the component-wise deviation from the
	corresponding uniform distribution by
	\[
	\mathcal Q(P)
	=
	\sum_{j=1}^{n}
	\left(
	p_j^{(1)}
	-
	\frac{s^{(1)}(P)}{n}
	\right)^2\eone
	+
	\sum_{j=1}^{n}
	\left(
	p_j^{(2)}
	-
	\frac{s^{(2)}(P)}{n}
	\right)^2\etwo
	=
	\mathcal Q^{(1)}(P)\eone
	+
	\mathcal Q^{(2)}(P)\etwo.
	\]
	For fixed total mass, each component $\mathcal Q^{(\ell)}$ is Schur-convex,
	because it is obtained by summing the convex function
	$x\mapsto (x-c)^2$ over the entries, with $c$ fixed by the total mass.
	
	\begin{proposition}
		\label{prop:quadratic-deviation-contraction}
		Let $n>2$, and let $\Phi:\Dint{0}{1}\to\Dpos$ be an admissible
		hyperbolic generator satisfying the hypotheses of
		Theorem~\ref{thm:generated-negator-yager-dominated}. For
		$P\in\PD(n)$, define $P_m=\Neg_\Phi^m(P)$. Then, for every $m\ge0$,
		\[
		\mathcal Q(P_m)
		\preceqD
		\frac{1}{(n-1)^{2m}}\mathcal Q(P).
		\]
		In particular, for every $j=1,\ldots,n$ and every $\ell\in\{1,2\}$,
		\[
		\left(
		p_{m,j}^{(\ell)}
		-
		\frac{s^{(\ell)}(P)}{n}
		\right)^2
		\le
		\frac{1}{(n-1)^{2m}}
		\mathcal Q^{(\ell)}(P).
		\]
		Moreover,
		\[
		\Ddist\!\left(P_m,U_n^{s(P)}\right)^2
		\le
		\frac{1}{(n-1)^{2m}}
		\sqrt{\mathcal Q^{(1)}(P)\mathcal Q^{(2)}(P)}.
		\]
	\end{proposition}
	
	\begin{proof}
		We first establish a one-step estimate. Let $R\in\PD(n)$ and set
		$S=\Neg_\Phi(R)$. By Theorem~\ref{thm:generated-negator-yager-dominated}, we have $\Yager(R)\Dmaj S$.
		Both $\Yager$ and $\Neg_\Phi$ preserve total mass, so
		$s(\Yager(R))=s(S)=s(R)$.
		
		Fix an idempotent component $\ell$. If this component is inactive, then $\mathcal Q^{(\ell)}(S)=\mathcal Q^{(\ell)}(R)=0$.
		
		Suppose now that it is active. Since $\Yager(R)^{(\ell)}\maj S^{(\ell)}$, and both vectors have the same total mass, Schur-convexity of
		$\mathcal Q^{(\ell)}$ gives
		\[
		\mathcal Q^{(\ell)}(S)
		\le
		\mathcal Q^{(\ell)}(\Yager(R)).
		\]
		For the Yager hyperbolic negator, every active component satisfies
		\[
		\Yager(R)_j^{(\ell)}
		-
		\frac{s^{(\ell)}(R)}{n}
		=
		-\frac{1}{n-1}
		\left(
		r_j^{(\ell)}
		-
		\frac{s^{(\ell)}(R)}{n}
		\right).
		\]
		Therefore
		\[
		\mathcal Q^{(\ell)}(\Neg_\Phi(R))
		\le
		\mathcal Q^{(\ell)}(\Yager(R))
		=
		\frac{1}{(n-1)^2}
		\mathcal Q^{(\ell)}(R).
		\]
		
		We now apply induction on $m$. For $m=0$, the desired estimate is an
		equality. Assume that the estimate has been obtained for all indices up
		to $m$. Applying the one-step estimate to $R=P_m$ gives
		\[
		\mathcal Q^{(\ell)}(P_{m+1})
		\le
		\frac{1}{(n-1)^2}
		\mathcal Q^{(\ell)}(P_m).
		\]
		Using the induction hypothesis,
		\[
		\mathcal Q^{(\ell)}(P_{m+1})
		\le
		\frac{1}{(n-1)^2}
		\frac{1}{(n-1)^{2m}}
		\mathcal Q^{(\ell)}(P)
		=
		\frac{1}{(n-1)^{2(m+1)}}
		\mathcal Q^{(\ell)}(P).
		\]
		Thus, for every $m\ge0$ and every $\ell\in\{1,2\}$,
		\[
		\mathcal Q^{(\ell)}(P_m)
		\le
		\frac{1}{(n-1)^{2m}}
		\mathcal Q^{(\ell)}(P),
		\]
		which is equivalent to
		\[
		\mathcal Q(P_m)
		\preceqD
		\frac{1}{(n-1)^{2m}}\mathcal Q(P).
		\]
		
		Each squared deviation is bounded above by the sum of all squared
		deviations in the same component. Hence
		\[
		\left(
		p_{m,j}^{(\ell)}
		-
		\frac{s^{(\ell)}(P)}{n}
		\right)^2
		\le
		\mathcal Q^{(\ell)}(P_m)
		\le
		\frac{1}{(n-1)^{2m}}
		\mathcal Q^{(\ell)}(P).
		\]
		
		Finally, by the definition of $\Ddist$ and the Cauchy--Schwarz inequality,
		\[
		\begin{aligned}
			\Ddist\!\left(P_m,U_n^{s(P)}\right)^2
			&=
			\sum_{j=1}^{n}
			\left|
			\left(
			p_{m,j}^{(1)}
			-
			\frac{s^{(1)}(P)}{n}
			\right)
			\left(
			p_{m,j}^{(2)}
			-
			\frac{s^{(2)}(P)}{n}
			\right)
			\right|\\
			&\le
			\sqrt{
				\mathcal Q^{(1)}(P_m)
				\mathcal Q^{(2)}(P_m)
			}\\
			&\le
			\frac{1}{(n-1)^{2m}}
			\sqrt{
				\mathcal Q^{(1)}(P)
				\mathcal Q^{(2)}(P)
			}.
		\end{aligned}
		\]
	\end{proof}
	
	\begin{theorem}
		\label{thm:generated-negator-converges-to-uniform}
		Let $n>2$, and let $\Phi:\Dint{0}{1}\to\Dpos$ be an admissible
		hyperbolic generator satisfying the hypotheses of
		Theorem~\ref{thm:generated-negator-yager-dominated}. For
		$P\in\PD(n)$, define $P_m=\Neg_\Phi^m(P)$. Then
		\[
		\lim_{m\to\infty}p_{m,j}^{(\ell)}
		=
		\frac{s^{(\ell)}(P)}{n}
		\]
		for every $j=1,\ldots,n$ and every $\ell\in\{1,2\}$.
	\end{theorem}
	
	\begin{proof}
		Fix an index $j$ and an idempotent component $\ell$. By
		Proposition~\ref{prop:quadratic-deviation-contraction},
		\[
		\left(
		p_{m,j}^{(\ell)}
		-
		\frac{s^{(\ell)}(P)}{n}
		\right)^2
		\le
		\frac{1}{(n-1)^{2m}}
		\mathcal Q^{(\ell)}(P).
		\]
		Therefore
		\[
		0
		\le
		\left|
		p_{m,j}^{(\ell)}
		-
		\frac{s^{(\ell)}(P)}{n}
		\right|
		\le
		\frac{1}{(n-1)^m}
		\sqrt{\mathcal Q^{(\ell)}(P)}.
		\]
		Since $n>2$, the factor $1/(n-1)$ is strictly smaller than $1$, and hence
		\[
		\lim_{m\to\infty}
		\frac{1}{(n-1)^m}
		\sqrt{\mathcal Q^{(\ell)}(P)}
		=
		0.
		\]
		By the squeeze theorem for real sequences,
		\[
		\lim_{m\to\infty}
		\left|
		p_{m,j}^{(\ell)}
		-
		\frac{s^{(\ell)}(P)}{n}
		\right|
		=
		0.
		\]
		Thus
		\[
		\lim_{m\to\infty}
		p_{m,j}^{(\ell)}
		=
		\frac{s^{(\ell)}(P)}{n}.
		\]
		Since $j$ and $\ell$ were arbitrary, the proof is complete.
	\end{proof}

	We conclude the section with two iterative examples. First illustrates the theory for a concave generator satisfying the hypotheses above. The second shows that monotonicity and the boundary condition $\varphi(1)=0$ are not sufficient by themselves without the structural assumptions that ensure the comparison principle. The entropy profiles and the hyperbolic separation from the uniform distribution need not exhibit the same uniforming behavior.
	
	\begin{example}
		Let $n=4$. We consider the generated hyperbolic negator associated with
		the component-wise generator defined on $[0,1]$ by
		\[
		\varphi(t)=1-t^2.
		\]
		The function $\varphi$ is non-increasing and concave, and satisfies
		$\varphi(1)=0$. Hence it belongs to the concave class of generators
		considered above.
		
		For $(P_r)_m=(P_r)_m^{(1)}\eone+(P_r)_m^{(2)}\etwo$, the iteration
		$(P_r)_{m+1}=\Neg_{\varphi}((P_r)_m)$ is computed component-wise. For
		every $j=1,\ldots,4$ and every active idempotent component $\ell$, we have
		\[
		\Neg_{\varphi}((P_r)_m)^{(\ell)}_j
		=
		\frac{1-\bigl((p_r)^{(\ell)}_{m,j}\bigr)^2}
		{\sum_{i=1}^{4}\left(1-\bigl((p_r)^{(\ell)}_{m,i}\bigr)^2\right)}.
		\]
		
		We illustrate the first iterates of this negator from the following
		hyperbolic-valued probability distributions:
		\[
		\begin{array}{r@{\;=\;}c@{\;}c@{\;+\;}c@{\;}c}
			P_1 & (1,0,0,0)              & \eone & (0,1,0,0)               & \etwo,\\[2pt]
			P_2 & (1,0,0,0)              & \eone & (0.75,0.25,0,0)         & \etwo,\\[2pt]
			P_3 & (0.75,0.25,0,0)        & \eone & (0,1,0,0)               & \etwo,\\[2pt]
			P_4 & (0.90,0.10,0,0)        & \eone & (0.50,0.25,0.25,0)      & \etwo,\\[2pt]
			P_5 & (0.50,0.25,0.25,0)     & \eone & (0.90,0.10,0,0)         & \etwo,\\[2pt]
			P_6 & (0.60,0.30,0.10,0)     & \eone & (0.45,0.30,0.20,0.05)   & \etwo.
		\end{array}
		\]
		Figure~\ref{fig:negator-iteration-profiles} shows the corresponding
		iterative profiles. In panels~(a) and~(b), the arrows represent the
		transitions $(P_r)_m\mapsto (P_r)_{m+1}$ in the planes
		$\bigl(H_{\mathbb D}^{(1)},H_{\mathbb D}^{(2)}\bigr)$ and
		$\bigl(G_{\mathbb D}^{(1)},G_{\mathbb D}^{(2)}\bigr)$, respectively.
		Panel~(c) shows the scalar profile of the hyperbolic separation
		$d_{\mathbb D}((P_r)_m,U_4^1)$ as a function of the iteration step $m$.
		
		The behavior shown in Figure~\ref{fig:negator-iteration-profiles} is
		consistent with the entropy-increasing and uniformizing conclusions
		proved above. The entropy profiles move toward the maximal values
		$(\log 4,\log 4)$ and $(3/4,3/4)$. At the same time, the hyperbolic separation from the uniform distribution tends toward zero along the displayed iterations.
		
		\begin{figure}[h]
			\centering
			\includegraphics[width=\textwidth]{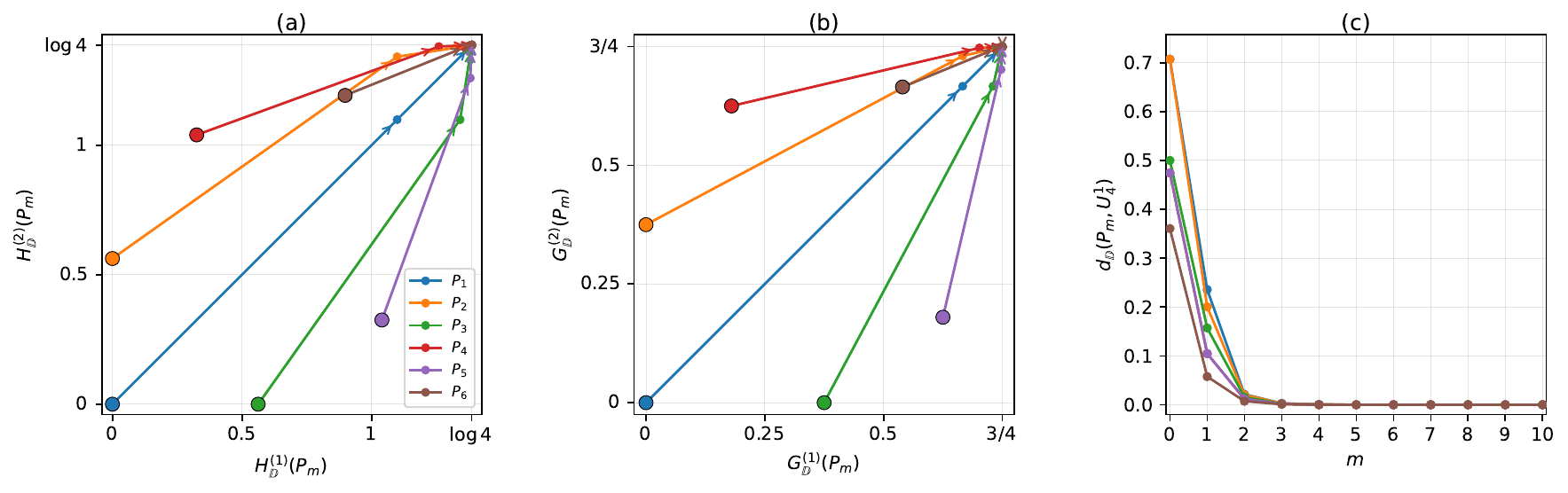}
			\caption{Iterative profiles of the generated hyperbolic negator induced by
				$\varphi(t)=1-t^2$, with $n=4$. Panels (a) and (b) show the component-wise
				evolution of the strong hyperbolic Shannon entropy and the hyperbolic
				Gini--Simpson entropy. The arrows represent the transitions
				$(P_r)_m\mapsto (P_r)_{m+1}$. Panel (c) shows the hyperbolic separation
				$d_{\mathbb D}((P_r)_m,U_4^1)$ as a function of the iteration step.}
			\label{fig:negator-iteration-profiles}
		\end{figure}
		
	\end{example}
	
	\begin{example}
		Let $n=4$. We now consider the generated hyperbolic negator associated with
		the component-wise generator defined on $[0,1]$ by
		\[
		\psi(t)=(1-t)^5.
		\]
		This generator is non-increasing and satisfies $\psi(1)=0$, but it does not
		belong to the concave class considered above. The purpose of this example is
		to illustrate that monotonicity and the boundary condition at $1$ are not
		sufficient by themselves to guarantee the entropy-increasing and
		uniformizing behavior obtained in the previous example.
		
		\begin{figure}[h]
			\centering
			\includegraphics[width=\textwidth]{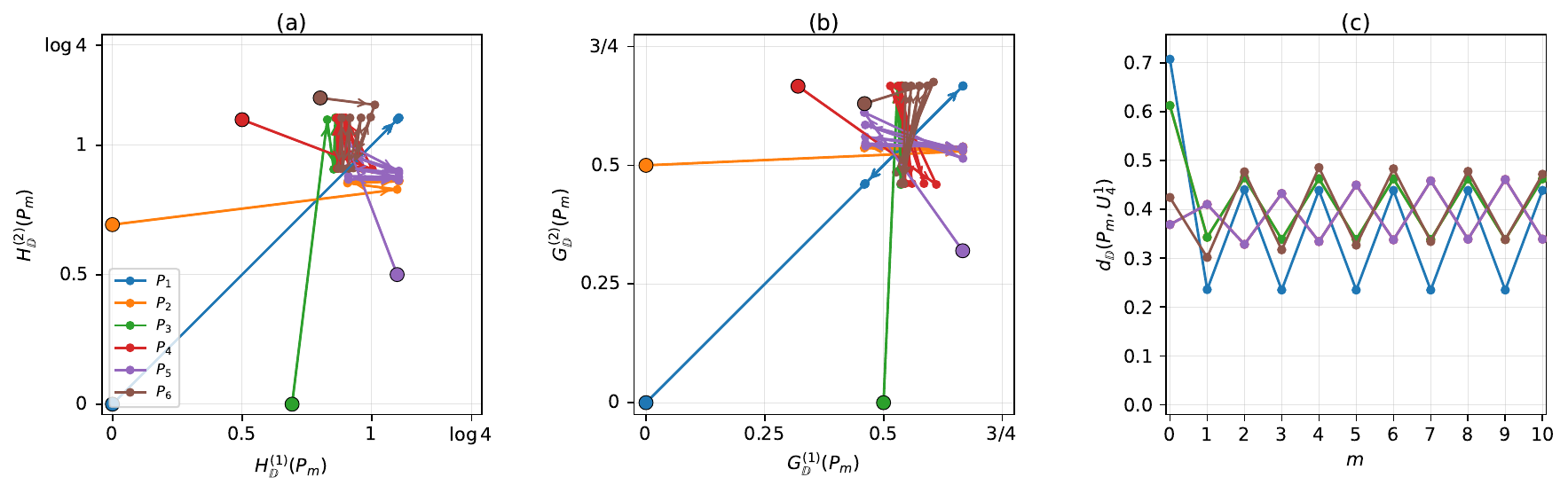}
			\caption{Iterative profiles of the generated hyperbolic negator induced by
				$\psi(t)=(1-t)^5$, with $n=4$. Panels (a) and (b) show the component-wise
				evolution of the strong hyperbolic Shannon entropy and the hyperbolic
				Gini--Simpson entropy. The arrows represent the transitions
				$(P_r)_m\mapsto (P_r)_{m+1}$. Panel (c) shows the hyperbolic separation
				$d_{\mathbb D}((P_r)_m,U_4^1)$ as a function of the iteration step.}
			\label{fig:negator-iteration-nonconcave}
		\end{figure}
		
		We consider the hyperbolic-valued probability distributions:
		\[
		\begin{array}{r@{\;=\;}c@{\;}c@{\;+\;}c@{\;}c}
			P_1 & (1,0,0,0)             & \eone & (0,1,0,0)            & \etwo,\\[2pt]
			P_2 & (1,0,0,0)             & \eone & (0.50,0.50,0,0)      & \etwo,\\[2pt]
			P_3 & (0.50,0.50,0,0)       & \eone & (0,0,1,0)            & \etwo,\\[2pt]
			P_4 & (0.80,0.20,0,0)       & \eone & (0.34,0.33,0.33,0)   & \etwo,\\[2pt]	
			P_5 & (0.34,0.33,0.33,0)    & \eone & (0.80,0.20,0,0)      & \etwo,\\[2pt]
			P_6 & (0.70,0.20,0.10,0)    & \eone & (0.55,0.15,0.15,0.15)& \etwo.
		\end{array}
		\]
		Figure~\ref{fig:negator-iteration-nonconcave} shows the corresponding
		iterative profiles. In panels~(a) and~(b), the arrows represent the
		transitions $(P_r)_m\mapsto (P_r)_{m+1}$ in the entropy planes
		$\bigl(H_{\mathbb D}^{(1)},H_{\mathbb D}^{(2)}\bigr)$ and
		$\bigl(G_{\mathbb D}^{(1)},G_{\mathbb D}^{(2)}\bigr)$. Panel~(c) shows
		the hyperbolic separation $d_{\mathbb D}((P_r)_m,U_4^1)$ as a function of
		the iteration step $m$.
		
		The behavior shown in Figure~\ref{fig:negator-iteration-nonconcave}
		contrasts with the behavior obtained for the concave generator
		$\varphi(t)=1-t^2$. The entropy profiles are not forced toward the maximal
		points $(\log 4,\log 4)$ and $(3/4,3/4)$, and the hyperbolic separation
		does not display the same uniformizing pattern. This illustrates the role
		of the structural hypotheses imposed on the generator in the entropy and
		uniformization results.
	\end{example}

	\section{Conclusions}
	
	This paper introduced a framework for studying negation on finite hyperbolic-valued probability distributions. The construction relies on the idempotent decomposition of hyperbolic numbers and on the partial order induced by their positive cone. Within this setting, we defined hyperbolic majorization as a component-wise criterion for comparing dispersion, and we introduced hyperbolic negators as mass-preserving transformations that reverse the hyperbolic order whenever comparison is available.
	
	The main results identify a class of generated hyperbolic negators for which negation moves distributions toward greater dispersion. More precisely, under suitable structural assumptions on the generator, the original distribution majorizes its negation. This comparison implies monotonicity of the strong hyperbolic Shannon entropy and of the hyperbolic Gini--Simpson entropy. It also gives a quantitative uniformization consequence, since the component-wise deviations from the corresponding uniform distribution decay geometrically.
	
	These results show that the hyperbolic framework preserves the classical intuition that negation should increase uncertainty, but does so in a genuinely ordered and component-wise setting. The partial order is essential here. It determines when order reversal is meaningful, while the idempotent decomposition allows the use of real majorization and Schur-concavity on each active component. Thus, the theory is not simply a formal duplication of real-valued probability negation, but an extension in which mass strata, inactive components, and partial comparability must be taken into account.
	
	We also showed that generated negators are generally not involutive under natural hypotheses, and we presented an involutive construction based on level reversal. This indicates that involutivity is compatible with the hyperbolic framework, although it belongs to a different structural mechanism from the generated negators used to obtain entropy increase and component-wise uniformization.
	
	Several directions remain open and advances are in progress to be communicated elsewhere. In particular, to characterize larger classes of generators for which the majorization comparison holds, and to determine whether the sufficient conditions used here can be weakened. Another direction is to study additional hyperbolic uncertainty measures and identify which of them are monotone under hyperbolic negation. Finally, to extend the theory beyond finite distributions, for instance to measurable hyperbolic-valued probability spaces.
	
	From an applied perspective, hyperbolic-valued probabilities may be useful in models where uncertainty has two coupled components or where one component may become inactive. In such settings, hyperbolic negators could provide complementary transformations for decision-making, information fusion, evidence-based reasoning, and uncertainty modeling. The results of this paper provide a mathematical basis for such developments by ensuring that, under explicit hypotheses, negation behaves as an uncertainty-increasing transformation.

\end{document}